\catcode`\@=11
\documentclass[11pt]{amsart}

\usepackage{sasha}
\usepackage{sasha_prih}
\usepackage{rank1}

\usepackage{amsthm}  

\newtheorem{thm}{Theorem} 
\newtheorem{lem}[thm]{Lemma}

\newtheorem{cor}[thm]{Corollary}

\theoremstyle{definition}
\newtheorem{defn}[thm]{Definition}

\newtheorem{quest}[thm]{Question}

\newtheorem{exmp}[thm]{Example}

\theoremstyle{remark}
\newtheorem{rem}[thm]{Remark}

\usepackage{color}
\usepackage{graphicx}

\title[On ergodic properties of iceberg transformations]
	{On ergodic properties of ``iceberg'' transformations. \\ I:~Approximation and spectral multiplicity}

\author{A.\,A.\,Prikhod'ko}

\begin{document}

\begin{abstract}
We investigate a class of mixing dynamical systems 
around the concept of iceberg transformation. 
In brief, an {\it iceberg transformation\/} is defined using symbolic language as follows. 
We build a sequence of words such that the next word is 
a concatenation of rotated copies of the previous word. 
For example, a word CAT can turn into CAT.ATC.TCA.TCA.CAT.ATC, then we repeat the procedure 
applying it to this new word and so on.
Geometrically, given an invertible measure preserving transformation $T$ 
an {\it iceberg\/} is a union of two icelets for the map~$T$, one direct and one reverse 
with common base set, where {\it icelet\/} is defined in a similar way as Rokhlin tower 
${B \sqcup TB \sqcup \ldots \sqcup T^{h-1}B}$, namely, 
an {\it icelet\/} is a sequence of disjoint measurable sets 
${\{B_0, B_1, \ldots, B_{h-1}\}}$ such that 
${B_{j+1} \subseteq TB_j}$, in other words, the levels $B_j$ continuously disappear 
from the base to the top of the icelet. Reverse icelet is defined as icelet for the inverse map $T^{-1}$, 
and it ``grows'' towards the past. Iceberg transformation is approximated by a sequence of icebergs, 
resembling the behaviour of rank one ergodic maps. 
The definition of iceberg maps essentially involves the notion of 
{\it interval exchange transformation}. 

We study combinatorial and ergodic properties for several classes of measure preserving transformations 
satisfying iceberg approximation including random iceberg maps and explicitely definied iceberg maps, 
in particular, involving finite fields. 
We also consider a series of extensions for iceberg approximation property. 
It~is a~common phenomenon that iceberg approximation implies local rank property and, hence, 
finite multiplicity of spectrum. 
It is show that a class of iceberg transformations 
almost surely has simple spectrum, $1/4$-local rank property 
and spectral type $\sigma$ such that ${\sigma \conv \sigma \ll \la}$ 
where $\la$ is the Lebesgue measure on the circle~$S^1$. 
%
%
%
%

%
\end{abstract}

\newcounter{nfigure}[section]
\renewcommand{\thenfigure}{\thesection.\arabic{nfigure}}

\maketitle


\begin{flushright}
	\it In memory of V.\,I.\,Arnold 
\end{flushright}

\section{Iceberg maps at a glance} \label{sIcebergAtAGlance} 

Let us define {\it rotation operator\/} $\rho_\a$ on finite words: 
if ${W = W_{(1)}W_{(2)}}$ and the length of the first subword ${|W_{(1)}| = \a}$ 
then we set ${\rho_\a(W) = W_{(2)}W_{(1)}}$. 
Observe that in other terms $\rho_\a$ cuts the word $W$ after $\a$ positions 
and then substitutes $W_{(1)}$ and~$W_{(2)}$. This kind of transform is a discrete 
variation of the well-known {\it interval exchange map}. 
Starting from a word $W_0$ consider the following {\it rotated words concatenation\/} procedure. 
A~word $W_n$ is repeated $q_n$ times, next, each copy is rotated 
by given value of positions $\a_{n,y}$, and the next word in the sequence is given by the formula 
\begin{equation}\label{eDefWnPlusOne}
	W_{n+1} = \rho_{\a_{n,0}}(W_n)\rho_{\a_{n,1}}(W_n)\ldots \rho_{\a_{n,{q_n-1}}}(W_n). 
\end{equation}
For example, if $W_1$ is the word ``CAT'', ${q_1 = 6}$ and 
%
$
	(\rho_{\a_{1,0}}, \rho_{\a_{1,1}}, \ldots, \rho_{\a_{1,{q_n-1}}}) = (0,1,2,2,0,1) 
$, 
%
then 
\begin{equation}
	\CAT \mapsto \CAT.\ATC.\TCA.\TCA.\CAT.\ATC = W_2 
\end{equation}
(points ``.'' are used to distinguish groups of symbols). 
At the next step we rotate the word $W_2$. The following table shows positions of cutting ($\times$) 
\begin{align}
	{\mathrm{ CATATCT_\times CATCACATATC }} \\
	{\mathrm{ CATA_\times TCTCATCACATATC }} \notag \\
	{\mathrm{ CATATCTCATC_\times ACATATC }} \notag 
\end{align}
used to create the word 
\begin{equation}
	W_3 = {\scriptstyle{\mathrm{ 
		CATCACATATC \,|\, CATATCT \:\seppoint\: 
		TCTCATCACATATC \,|\, CATA \:\seppoint\: 
		ACATATC \,|\, CATATCTCATC
		}}} \ldots
\end{equation}
It is a common phenomenon that this sequence of words generates a dynamical system. 
Simply speaking the words in the sequence $W_n$ become more and more stationary according 
to empirical distributions on words of bounded length, and the dynamical system is 
associated with the shift map ${T \Maps (x_n) \mapsto (x_{n+1})}$. 
At~the same time there is a~simple way to produce a geometrical 
description of the dynamics for this symbolic system. Indeed, one way of drawing $\rho_a$ 
is to fix cut points, like it is shown in the following line 
\begin{equation}\label{eRotatedCATs} 
		{\mathrm{ 
		CAT_\times					, \quad 
		C_\times AT 				, \quad 
		CA_\times T 				, \quad 
		CA_\times T					, \quad 
		CAT_\times					, \quad 
		C_\times AT  
		}}
\end{equation} 
which implies mapping 
\begin{equation}\label{eCATdefWithCuts}  
	\CAT \quad\mapsto\quad {\mathrm{ 
		CAT					\,.\,
		ATC					\,.\,
		TCA					\,.\,
		TCA					\,.\, 
		CAT					\,.\, 
		ATC
		}},
\end{equation} 
and other way is to think that the word $\CAT$ is shifted as a function on the group $\Set{Z}_h$ 
(in this case shift actually coincides with the rotation oeprator~$\rho_\a$). 
\begin{figure}[th]
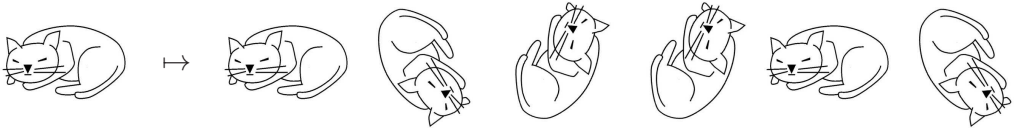

$$	\gCAT \quad \raise 8mm\hbox{$\mapsto$} \quad \gCAT \gATC \gTCA \gTCA \gCAT \gATC $$
  \caption{Each entrance of the word $\CAT$ in $W_2$ is rotated by the 
  map ${t \mapsto t+\a_{1,y} \pmod{3}}$.} 
  \label{fGrRotatingCATs}
\end{figure}

Suppose now that symbols $\CATalphabet$ correspond to a partition $\{P_\syC,P_\syA,P_\syT\}$ 
of the probability space $(X,\cA,\mu)$ (we assume that $X$ is a Lebesgue space without atoms 
which is isomorphic by Rokhlin's theorem to the unit segment $[0,1]$ with Lebesgue measure), 
${X = P_\syC \sqcup P_\syA \sqcup P_\syT}$ and ${\mu(P_\syC) = \mu(P_\syA) = \mu(P_\syT) = \frac13}$.

\begin{figure}[th]
  \centering
  \unitlength=1mm
  \begin{picture}(96,56)
      \footnotesize
      \put(30,21)
      {
      	\put(20,0)
      	{
        	\linethickness{1pt}
        	\put(0,0){\line(1,0){10}}
        	\put(0,0){\line(0,1){30}}
        	\put(0,30){\line(1,0){10}}
        	\put(10,0){\line(0,1){30}}
        	\put(3.3,33){$V_3$}
        	\put(13,3){$B_0$}
        	\put(13,13){$B_1$}
        	\put(13,23){$B_2$}
        	\put(3,3){\Large $\syC$}
        	\put(3,13){\Large $\syA$}
        	\put(3,23){\Large $\syT$}
        	\linethickness{0.2pt}
        	\put(0,10){\line(1,0){10}}
        	\put(0,20){\line(1,0){10}}
        }
      	\put(10,-10)
      	{
        	\linethickness{1pt}
        	\put(0,0){\line(1,0){10}}
        	\put(0,0){\line(0,1){30}}
        	\put(0,30){\line(1,0){10}}
        	\put(10,0){\line(0,1){30}}
        	\put(3.3,33){$V_2$}
        	\put(3,3){\Large $\syT$}
        	\put(3,13){\Large $\syC$}
        	\put(3,23){\Large $\syA$}
        	\linethickness{0.2pt}
        	\put(0,10){\line(1,0){10}}
        	\put(0,20){\line(1,0){10}}
        }
      	\put(0,-20)
      	{
        	\linethickness{1pt}
        	\put(0,0){\line(1,0){10}}
        	\put(0,0){\line(0,1){30}}
        	\put(0,30){\line(1,0){10}}
        	\put(10,0){\line(0,1){30}}
        	\put(3.3,33){$V_1$}
        	\put(-7.8,3){$B_{-2}$}
        	\put(-7.8,13){$B_{-1}$}
        	\put(-7.8,23){$B_0$}
        	\put(3,3){\Large $\syA$}
        	\put(3,13){\Large $\syT$}
        	\put(3,23){\Large $\syC$}
        	\linethickness{0.2pt}
        	\put(0,10){\line(1,0){10}}
        	\put(0,20){\line(1,0){10}}
        }
      \put(20,0)
      {
      }
      \put(10,-10)
      {
      }
      \put(0,-20)
      {
      }
      }
  \end{picture}
  \caption{Iceberg associated with the word $\CAT$. 
  	Letters should be read in vertical direction. 
		The~map~$T$ lifts most part of any elementary set (a square) on the picture to the upper level.} 
  \label{fCAT}
\end{figure}
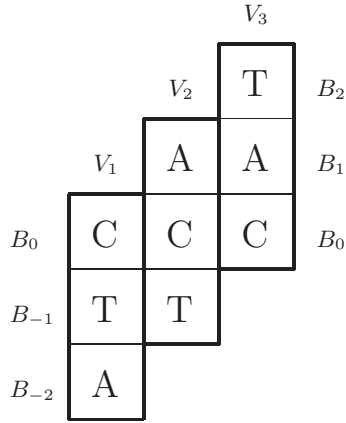

\begin{defn}
Let us draw at fig.\,\ref{fCAT} all the rotations of the word $\CAT$ in a way used to draw a Rokhlin tower, 
placing same letters to the same level and placing up the letter presumed to be next 
after the letter of the current level. This picture is called {\it iceberg}. 
%
%
The level corresponding to the first letter of the word (letter ``$\syC$'') 
is called the {\it base level\/} of the iceberg. 
\end{defn}

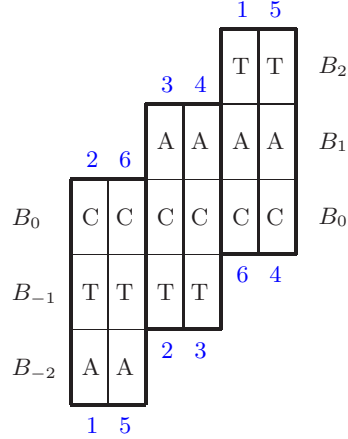
\begin{figure}[th]
  \centering
  \unitlength=1mm
  \begin{picture}(96,53)
      \footnotesize
      \put(30,21)
      {
      	\put(20,0)
      	{
        	\linethickness{1pt}
        	\put(0,0){\line(1,0){10}}
        	\put(0,0){\line(0,1){30}}
        	\put(0,30){\line(1,0){10}}
        	\put(10,0){\line(0,1){30}}
        	\linethickness{0.6pt}
        	\put(5,0){\line(0,1){30}}
        	\put(13,4.1){$B_0$}
        	\put(13,14.1){$B_1$}
        	\put(13,24.1){$B_2$}
        	\put(1.5,4){$\syC$}
        	\put(1.5,14){$\syA$}
        	\put(1.5,24){$\syT$}
        	\put(6,4){$\syC$}
        	\put(6,14){$\syA$}
        	\put(6,24){$\syT$}
        	\linethickness{0.2pt}
        	\put(0,10){\line(1,0){10}}
        	\put(0,20){\line(1,0){10}}
        }
      	\put(10,-10)
      	{
        	\linethickness{1pt}
        	\put(0,0){\line(1,0){10}}
        	\put(0,0){\line(0,1){30}}
        	\put(0,30){\line(1,0){10}}
        	\put(10,0){\line(0,1){30}}
        	\linethickness{0.6pt}
        	\put(5,0){\line(0,1){30}}
        	\put(1.5,4){$\syT$}
        	\put(1.5,14){$\syC$}
        	\put(1.5,24){$\syA$}
        	\put(6,4){$\syT$}
        	\put(6,14){$\syC$}
        	\put(6,24){$\syA$}
        	\linethickness{0.2pt}
        	\put(0,10){\line(1,0){10}}
        	\put(0,20){\line(1,0){10}}
        }
      	\put(0,-20)
      	{
        	\linethickness{1pt}
        	\put(0,0){\line(1,0){10}}
        	\put(0,0){\line(0,1){30}}
        	\put(0,30){\line(1,0){10}}
        	\put(10,0){\line(0,1){30}}
        	\linethickness{0.6pt}
        	\put(5,0){\line(0,1){30}}
        	\put(-7.8,3.9){$B_{-2}$}
        	\put(-7.8,13.9){$B_{-1}$}
        	\put(-7.8,23.9){$B_0$}
        	\put(1.5,4){$\syA$}
        	\put(1.5,14){$\syT$}
        	\put(1.5,24){$\syC$}
        	\put(6,4){$\syA$}
        	\put(6,14){$\syT$}
        	\put(6,24){$\syC$}
        	\linethickness{0.2pt}
        	\put(0,10){\line(1,0){10}}
        	\put(0,20){\line(1,0){10}}
        }
      	\put(20,0)
      	{
      		\color{blue}
        	\put(1,31.6){$1$}
        	\put(5.4,31.6){$5$}
        	\put(1,-3.7){$6$}
        	\put(5.4,-3.7){$4$}
        }
      	\put(10,-10)
      	{
      		\color{blue}
        	\put(1,31.6){$3$}
        	\put(5.4,31.6){$4$}
        	\put(1,-3.7){$2$}
        	\put(5.4,-3.7){$3$}
        }
      	\put(0,-20)
      	{
      		\color{blue}
        	\put(1,31.6){$2$}
        	\put(5.4,31.6){$6$}
        	\put(1,-3.7){$1$}
        	\put(5.4,-3.7){$5$}
        }
      }
  \end{picture}
  \caption{Poincar\'e map for the iceberg corresponding to sequence $\CAT.\ATC.\TCA.\TCA.\CAT.\ATC$} 
  \label{fdynCAT}
\end{figure}

Each column at fig.\,\ref{fCAT} corresponds to a rotation of the word $\CAT$: 
$V_1$~to~$\ATC$, $V_2$~to~$\TCA$ and $V_3$~to~$\CAT$. 

The~basic idea of iceberg is to guess that a measure preserving transformation $T$ 
maps each elementary set (shown as a square) to the upper set with small deviation. 
For example, if we split the column $V_3$ into letter-marked sets, 
${V = V_{3,\syC} \sqcup V_{3,\syA} \sqcup V_{3,\syT}}$, then we require: 
\begin{equation}
	\mu(TV_{3,\syC} \mid V_{3,\syA}) \approx 1 \quad \text{and} \quad \mu(TV_{3,\syA} \mid V_{3,\syT}) \approx 1. 
\end{equation}
Remark that if we denote levels of the iceberg as ${B_{-2},B_{-1},B_0,B_1,B_2}$ then 
${P_\syC = B_0}$, ${P_\syA = B_1 \cup B_{-2}}$, ${P_\syT = B_2 \cap B_{-1}}$, 


At this point it is completely unknown how $T$ acts on the top elementary set of each column. 
We specify the dynamics of~$T$ with the help of equation \eqref{eCATdefWithCuts} having the following 
translation to geometric language. 
We~divide each column $V_k$ into several vertical subcolumns 
to get one separate subcolumn for one entrance of rotated word ``$\CAT$'' in \eqref{eCATdefWithCuts} 
like it is shown on fig.\,\ref{fdynCAT}. 
Then we link the corresponding sets on the boundary marked by indexes $1,\ldots,6$ on fig.\,\ref{fdynCAT}. 
%
After linking we get the next (cyclic) sequence 
$\CAT.\ATC.\TCA.\TCA.\CAT.\ATC$ and we repeat the procedure. 
This construction is uniquely determined by rotations $(\a_{n,0},\ldots,\a_{n,q_n-1})$ for each step. 
Consider the {\it edge\/} of the iceberg, the union $E$ of all bottom sets of columns, 
and define the {\it Poincar\'e map\/} ${\tilde T_E \Maps E \to E}$, a measure preserving 
map corresponding to the way of linking subcolumns (see fig.\,\ref{fdynCAT}). 
Actually the map $\tilde T_E$ coincides with the return map $T_E$ induced by $T$ on~$E$ 
up to next step cuttings. 
Iterating this procedure we get a measure preserving transformation on a Lebesgue space. 


\medskip
Let us consider the space $L^2(X,\mu)$ of measurable functions 
${f \Maps X \to \Set{C}}$ with integrable squaare 
and define {\it Koopman operator} 
\begin{equation}
	\Hat T \Maps L^2(X,\mu) \to L^2(X,\mu) \Maps f(x) \to f(Tx). 
\end{equation}
The meaning of $\hat T$ is translation by $1$ step along the trajectory of~$T$. 
It can be easily seen that $\hat T$ is a unitary operator in space $L^2(X,\mu)$. 
By spectral teorem $\hat T$ is determined up to unitary equivalence by two invariants: 
spectral type $\sigma$ (a measure on $S^1$ up to equivalence) and multiplicity function 
$\Mult_T(z)$ mapping $S^1$ to the set ${\Set{N} \sqcup \{\infty\}}$. Since ${\hat T 1 = 1}$ 
usually $\hat T$ is restricted to the space of functions with zero mean ${\{ f \where \int f\,d\mu = 0\}}$. 
An~operator $\hat T$ has {\it simple spectrum\/} if there exist 
a function $e_0$ ({\it cyclic vector\/}) such that the iterations $T^k e_0$ 
generates the whole $L^2(X,\mu)$, where ${k \in \Set{Z}}$. 


\begin{thm}\label{thmRandomIcebergMap}
Let $T$ be an iceberg transformation given by 
uniform i.i.d.\ random rotations $\a_{n,k}$, and suppose that ${q_n \gg h_n}$ grows sufficiently fast. 
Then the following properties hold almost surely 
  \begin{itemize}
  	\item[(i)] $T$ has $1/4$-local rank (see definition in section 3), 
  	\item[(ii)] $\hat T$ has simple spectrum, 
  	\item[(iii)] $\sigma \conv \sigma \ll \la$, where $\sigma$ is the spectral type of $\hat T$  
  	 							and $\la$ is Lebesgue measure on~$S^1$, 
  	\item[(iv)] For a dense set of functions $f$ with zero mean $\forall \varepsilon > 0$ 
  								$$\langle T^t f,f \rangle = O(t^{-1/2+\varepsilon}).$$
  \end{itemize}
\end{thm}

In the first part of the paper we discuss combinatorial properties of iceberg transformations 
and prove statements (i) and (ii) of theorem~\ref{thmRandomIcebergMap}. 
We place in the second part the detailed investigation fo correlation decay, statements (iii) and~(iv). 
And the third part is devoted to a class of iceberg transformation involving finite field arithmetics.

\section{Motivation}

\subsection{Approximation and spectral invariants}\label{sMotivDynSysApprox}

Throughout this paper we consider invertible measure preserving tramsformations of 
a Lebesgue space $(X,\cA,\mu)$ as well as measure preserving group actions. 
A~measurable map ${T \Maps X \to X}$ is called {\it measure preserving\/} if 
${\mu(T^{-1}A) = \mu(A)}$ for any set ${A \in \cA}$. 
If $T$ is invertible and both $T$ and $T^{-1}$ are measure preserving then 
$T$ is usually called an {\it automorphism\/} on space $(X,\cA,\mu)$. 

A serious part of research in spectral theory of dynamical systems 
is based on the idea of approximation. 
The method and the first examples \cite{LemRankMorse,Oseledec1,Oseledec2,StepinDiss} 
originate from paper \cite{KSt} by A.\,Katok and A.\,St\"epin 
where the concept of {\it periodic approximation\/} was introdued. 
The method of approximation is also actively used in smooth dynamics 
(D.\,Anoson and A.\,Katok~\cite{AnosovKatok}, E.\,Sataev \cite{Sataev}, 
 A.\,Kochergin~\cite{KocherginAMix1}). 
Further development of this method has led to the notion of {\it rank one} approximation
(T.\,Adams and N.\,Friedman~\cite{Adams2}, R.\,Chacon~\cite{Chacon}, D.\,Ornstein~\cite{O})   
%
%
and a series of different concepts extending~it (see review by S.\,Ferenczi~\cite{Ferenczi1}).


The Koopman operator $\hat T$ associated with an automorphism $T$ 
is uniquely determinetd by the spectral type $\sigma$ and the multiplicity function $\Mult_T(z)$ 
(see~\cite{LemEncycloSpTh}). 
Denote $\spMult(T)$ the essential maximal value of the multiplicity function $\Mult_T(z)$. 
We say that $T$ has simple spectrum if it is of spectral multiplicity one, ${\spMult(T) = 1}$. 
Let us mention the following open question due to S.\,Banach.  

\begin{quest}
Does there exist an automorphism $T$ having simple spectrum 
and Lebesgue spectral type~$\sigma$\,? In other words, is it possible to find an automorphism $T$ 
such that for some ${e_0 \in L^2(X,\mu)}$ the sequence 
\begin{equation}
	\ldots,\ \hat T^{-1}e_0,\ e_0,\ Te_0,\ T^2e_0,\ \ldots 
\end{equation}
satisfies ${T^ie_0 \perp T^je_0}$ for ${i \not= j}$, 
and linear combinations of $T^je_0$ are dense in $L^2(X,\mu)$. 
\end{quest}

	  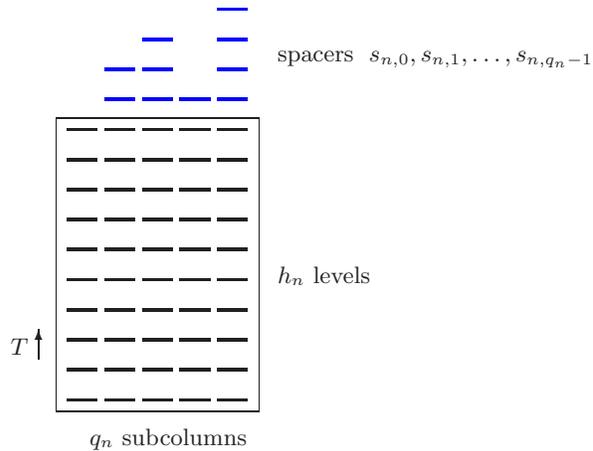
\begin{figure}[th]
	    \label{fStacking}
	    \centering
	    \unitlength=1mm
	    \begin{picture}(70,60)
	        \footnotesize
	        \put(20,6)
	        {
            \linethickness{1pt}
            \multiput(0,0)(5,0){5}{\line(1,0){4}} 
            \multiput(0,4)(0,4){9}{\line(1,0){4}}
            \multiput(5,4)(0,4){9}{\line(1,0){4}}
            \multiput(10,4)(0,4){9}{\line(1,0){4}}
            \multiput(15,4)(0,4){9}{\line(1,0){4}}
            \multiput(20,4)(0,4){9}{\line(1,0){4}}
            {\color{blue}
            \multiput(5,40)(0,4){2}{\line(1,0){4}}
            \multiput(10,40)(0,4){3}{\line(1,0){4}}
            \multiput(15,40)(0,4){1}{\line(1,0){4}}
            \multiput(20,40)(0,4){4}{\line(1,0){4}}
            }
            \linethickness{0.1pt}
            \put(-1.5,-1.5){\line(1,0){27}}
            \put(-1.5,-1.5){\line(0,1){39}}
            \put(-1.5,37.5){\line(1,0){27}}
            \put(25.5,-1.5){\line(0,1){39}}
            \put(-7.5,6){$T$}
            \linethickness{0.5pt}
            \put(-3.8,5.4){\vector(0,1){4}}
            %
            %
            \put(28,15.6){\footnotesize $h_n$ levels}
            \put(28,45){\footnotesize spacers $\; s_{n,0},s_{n,1},\ldots,s_{n,q_n-1}$}
            \put(3,-6){\footnotesize $q_{n}$ subcolumns}
	        }
	    \end{picture} 
	    \caption{Cutting-and-stacking construction of a rank one map} 
		\end{figure}

A {\it Rokhlin tower\/} of {\it height\/}~$h$ is a sequence of disjoint sets 
${\tow = \{C,TC,T^2C,\ldots,T^{h-1}C\}}$, where 
$T^jC$ are called {\it levels\/} of the tower, 
and $C$ is called the {\it base\/} set of~$\tow$. 
Denote ${\utw = C \sqcup TC \sqcup \ldots \sqcup T^{h-1}C}$.

\begin{defn}\label{defRankOne} 
A map $T$ is called {\it rank one\/} transformation 
if there exists a sequence of Rokhlin towers $\tow_n$ such that 
${\mu(\utw_n) \to 1}$ and the corresponding tower partitions $\twp_n$ 
approximate the $\sigma$-algebra $\cA$, in other words, 
for any measurable set $A$ there exist $\twp_n$-measurable%
\footnote{Given a tower $\tow = \{C,TC,\ldots,T^{h-1}C\}$ we use the same symbol $\twp$ 
 for the partition of the phase space into the levels $T^jC$ and the remainer set $X \sms \utw$}
sets $A_n$ with%
\footnote{$A \syms B$ denotes the symmetric difference $(A \sms B) \cup (B \sms A)$} 
${\mu(A_n \syms A) \to 0}$ as~${n \to \infty}$. 
\end{defn} 

\begin{thm}
Any rank one transformations $T$ 
has simple spectrum, ${\spMult(T) = 1}$. 
\end{thm}

In paper \cite{KwiatkowskiLemMultII} I.\,Kwiatkowski and M.\,Lema\'nczyk proved that for any subset $K$ of~$\Set{N}$ 
containing $1$ there exists a transformation $T$ having $K$ as a set of essential values 
of the multiplicity function $\Mult_T(z)$.

%
\subsection{Cutting-and-stacking construction}\label{sMotivRankOneGeom}
%
%
To build a rank one transformation 
we start from any tower $\tow_0$ and build a sequence of towers $\tow_n$. 
A the $n$-th step the next tower $\tow_{n+1}$ is constructed from the previous one as follows. 
We cut the tower $\tow_n$ in vertical direction into $q_n$ equal {\it columns\/}. 
Then we add $s_{n,y}$ extra levels ({\it spacers\/}) to the top of each column 
and stack these extended columns $C_{n,y}$ together ${\tow_{n+1} = C_{n,0} C_{n,1} \ldots C_{n,q_n-1}}$ 
to get next tower~$\tow_{n+1}$. 
In terms of dynamics this means that whenever a~point comes 
to the top of a column $C_{n,y}$, next time it goes to the bottom level 
of the right-hand side column~$C_{n,y+1}$. 
This procedure repeated infinitely many times leads to an ergodic transformation 
if requirement~\eqref{eFiniteMeasRankOne} is satisfied.

\subsection{Rank one systems: Mixing and spectral type}

D.\,Ornstein in \cite{O} introduced a class of rank one transformations with random spacers 
defined as follows: we set ${s_{n,y} = \a_{n,y+1}-\a_{n,y}}$ where ${\a_{n,y} \ll h_n}$ 
are i.i.d.\ random variables. J.\,Bourgain has shown that Ornstein transformations 
are of singular spectral type almost surely~\cite{Bourgain}. 
A large class of generalized Ornstein transformations was studied by El~H.~El~Abdalaoui and F.\,Parreau 
\cite{elAbdal1,elAbdal3,AbPaPr}. 
I.\,Klemes has proved singularity of spectral type for a class of staircase constructions \cite{Klemes} 
defined by the spacer sequence ${s_{n,y} = y}$. 
Rank one transformations of Ornstein type as well as staircase constructions are 
also examples of mixing transformations. We say that $T$ is mixing if 
${\mu(T^j A \cap B) \to \mu(A)\,\mu(B)}$ for all measurable sets $A$ and~$B$. 
T.\,Adams proving Smorodinsky's conjecture \cite{Adams1} has shown that staircase rank one maps 
with ${q_n = n}$ are mixing. This result was extended by D.\,Creutz and C.E.\,Silva \cite{CreutzSilva}
and A.\,Danilenko and V.\,Ryzhikov (e.g.\ see~\cite{DanilenkoRyzh}).

Let us remark that $\hat T$ has simple spectrum iff there exists 
${f \in L^2(X,\mu)}$ ({\it cyclic vector}) such that 
\begin{equation}
	L^2(X,\mu) = \overline{{\mathrm{Span}}}(\{\Hat T^kf \where k \in \Set{Z}\}). 
\end{equation}
In this case $\sigma_f \sim \sigma$, where $f$ is a cyclic vector and $\sigma_f$ is defined by the property 
\begin{equation}
	\int_{S^1} z^k \,d\sigma_f = \scpr<T^kf,f>. 
\end{equation}

The following question concerning the spectral type of rank one transformations is still open. 

\begin{quest}
Is the following true: The spectral type of any rank one transformation is 
singular with respect to the Lebesgue measure $\la$ on~$S^1$? 
\end{quest}



\subsection{Generalized Riesz products}

If a function $f \in L^2(X,\mu)$ is constant on the levels of $n$-th tower $\tow_n$ for a rank one transformation 
then we identify $f$ with a function ${f_{(n)} \Maps \Set{Z} \to \Set{C}}$, ${f_{(n)}(j) \equiv f|_{T^j B_{n,0}}}$, 
where $T^j B_{n,0}$ is the level with index~$j$ of $n$-th tower. 
Let us define polynomials 
\begin{equation}
	P_n(z) = \frac1{\sqrt{q_n}} \sum_{y=0}^{q_n-1} z^{\omega_n(y)} \in \cM_{q_n}, 
\end{equation}
where 
\begin{equation}
	\omega_n(y) = yh_n + \sum_{j < y} s_{n,j}. 
\end{equation}
The spectral measure $\sigma_f$ can be represented as an infinite product (up to a constant multiplier) 
which converges in weak topology \cite{Bourgain,ChoksiNadkarni,KlemesReinhold,AbPaPr}: 
\begin{equation}
	\sigma_f = |\Hat f_{(n_0)}|^2 \prod_{n=n_0}^{\infty} |P_n(z)|^2, 
\end{equation}

\subsection{Littlewood polynomials and flatness phenomenon}\label{sMotivLittlePolyFlatness}

Let us consider the following classes of polynomials introduced by J.\,Littlewood \cite{Littlewood} 
(see also \cite{ErdelyiLittlewoodType02}) 
\begin{gather}
	\cK_n = \Bigl\{ P(z) = {\scriptstyle \frac1{\sqrt{n+1}}} 
				\sum_{k=0}^{n} a_k z^k \where |a_k| \equiv 1 \Bigr\}, 
	\\
	\cL_n = \Bigl\{ P(z) = {\scriptstyle \frac1{\sqrt{n+1}}} 
				\sum_{k=0}^{n} a_k z^k \where a_k \in \{-1,\;1\} \Bigr\}, 	
	\\
	\cM_n = \left\{ P(z) = {\scriptstyle \frac1{\sqrt n}}
				(z^{\omega_1} + z^{\omega_2} +\ldots+ z^{\omega_n}) 
				\where \omega_j \in \Set{Z},\ \omega_j < \omega_{j+1} \right\}.
\end{gather}
%
Polynomilas in the classs $\cK_n$ are called polynomials with {\it unimodular coefficients\/}. 

\begin{quest}[J.\,Littlewood, 1966]\label{qLittlewood}
Is the following true? For any $\eps > 0$ there exists a polynomial 
${P(z) \in \cK_n}$ such that 
\begin{equation}
	\forall z \in S^1 \qquad \bigl| |P(z)| - 1 \bigr| < \eps. 
\end{equation}
\end{quest}

\begin{thm}[Kahane, 1980]
The answer to question~\ref{qLittlewood} is ``yes'' with the speed of~convergence 
\begin{equation}
	\eps_n = O(n^{-1/17} \sqrt{\ln n}). 
\end{equation}
\end{thm}

A new progress in explicit constructions of ultra-flat unimodular polynomials on~$S^1$ 
is achieved by J.\,Bourgain and E.\,Bombieri in~\cite{BourgainBom}. 
Passing to group~$\Set{R}$ let us define class 
\begin{equation}
	\cM^{\Set{R}}_n = \left\{ P(z) = {\scriptstyle \frac1{\sqrt n}}
				(z^{\omega_1} + z^{\omega_2} +\ldots+ z^{\omega_n}) 
				\where \omega_y \in \Set{R},\ \omega_y < \omega_{y+1} \right\}.
\end{equation}
It is shown in~\cite{LebesgueFlows} that 
the answer to the question on $L^1$-flatness is positive in class $\cM^{\Set{R}}$ 
if we understand it as {\it flatness on compact sets\/} in~$(0,\infty)$. 
It occurs that flat sums of such kind given by exponential frequency functions 
\begin{equation}
	\omega_y = \frac{n}{\eps^2} e^{\eps y/n} 
\end{equation} 
satisfy natural conditions needed to be a polynomial in the Riesz product for a rank one flow, 
which proves the existence of rank one flows with simple Lebesgue spectrum.

\begin{quest}[\SIGN{open}]
Can we see flatness in $\cL_n$ or $\cM_n$? 
\end{quest}

%
T.\,Downarowicz and Y.\,Lacroix \cite{Downar} has proved that 
if all continuous binary Morse systems have singular spectra then 
the merit factors of binary words are bounded (the Turyn's conjecture holds). 
In the work~\cite{Guenais} M.\,Guenais has shown that 
the positive answer to the Littlewood question in~$\cL_n$ 
is equivalent to the fact that a class of transformations 
given by Morse cocycles has a Lebesgue component in spectrum.

\subsection{Idea}\label{sMotivIdea}

In this paper we investigate a new class of dynamical systems%
\footnote{To be correct we should say {\it a new approximation property\/} instead of 
	{\it new class of dynamical systems\/} because, in fact, it is unknown how to distinguish it, 
	for example, to prove that there exists a transformation in this new class which is not~rank~one 
	(cf.\ questions~\ref{qIcebergAndRank}). 
}, 
a hybrid concept extending both Katok--St\"epin periodic approximation and rank one property. 
Actually the idea leading to the new approximation property is purely {\it analytic\/}. 
It is known that the spectral type of a rank one transformation 
(a~map approximated by a sequence of Rokhlin towers) is given by generalized Riesz product 
\begin{equation}
	\sigma \sim |\hat f_{(n_0)}| \cdot \prod_{n=1}^\infty |P_n(z)|^2 \,\la, 
\end{equation}
where ${|z| = 1}$ and 
$\la$ is the Lebesgue measure on the unit circle, 
and the main ingredient in this product is a sequence of trigonometric polynomials $P_n(z)$ 
with coefficients $0$ and~$1$ (see \cite{ChoksiNadkarni,KlemesReinhold,AbPaPr}). 
Let us denote this class of polynomials as~$\cM$, 
\begin{equation}
	\cM = \Bigl\{ 
		P(z) = \sum_{y=0}^{q-1} z^{\omega(y)} \where \omega(y) \in \Set{Z},\ \omega(y) < \omega(y+1) 
	\Bigr\}. 
\end{equation}
Thus, understanding analytic properties of the class~$\cM$ could be a key to 
spectral properties of rank one transformations. 
For example, a hypotetic {\it flatness property\/} for a polynomial ${P \in \cM}$ would 
help to find a rank one map with absolutely continuous component in the spectrum. 
A~polynomial $P(z)$ is $\eps$-{\it flat\/} if ${ \|q^{-1/2}|P| - 1\| < \eps}$ in some norm $\|\cdot\|$. 
If we deal with a rank one transformation several limitations to the class of polynomials turn out. 
Points $\omega(y)$ must be almost equividistant, ${\omega(y+1)-\omega(y) \sim h}$, 
and constructing a rank one map we cannot involve any polynomial in~$\cM$. 
The main idea of iceberg transformation is to extend the class of dynamical systems in order 
to make the underlying class of polynomials richer. 

\subsection{Rank one maps: Symbolic interpretation}\label{sMotivRanknOneSymb} 

It was discovered that rank one property can be expressed in purely combinatorial terms: 
$T$ is a rank one map if and only if a typical orbit for $T$ 
is $\eps$-covered by words which are $\eps$-close to a single word $W(\eps)$ 
for arbitrary $\eps$. 
Furthermore, rank one maps can be described using one of the three equivalent definitions: 
measure-theoretic (definition~\ref{defRankOne}), sybmolic and geometrical definitions see. 
Since iceberg transformations extend in a~sense rank one systems 
we will find the corresponding three parallel interpretations for this concept. 
%
%


\begin{defn}\label{defRankOneSymbolic} 
Consider a sequence of words $W_n$ in alphabet~$\Set{A}$ such that 
\begin{equation}
	W_{n+1} = W_n 1^{s_{n,1}} W_n 1^{s_{n,2}} W_n \ldots 1^{s_{n,q_n}} W_n, 
\end{equation}
where symbol ``$1$'' is used to create {\it spacers\/} between words. 
Suppose that 
\begin{equation}\label{eFiniteMeasRankOne} 
	\prod_{n=1}^\infty \frac{h_{n+1}}{q_n h_n} < \infty. 
\end{equation}
This sequence of words is the coding of an orbit starting from the base of $n$-th tower 
according to the partition associated with the alphabet~$\Set{A}$.  
\end{defn}

\section{Iceberg map: Formal definition}\label{sIcebergMapDef} 

The purpose of this section is to discuss the formal definition of iceberg transformations. 
Actually we consider the family of transformations given by the construction 
discussed in section~\ref{sIcebergAtAGlance}. 
At the same time an abstract definition of {\it iceberg approximation\/} is introduced 
(see def.\,\ref{defIcebergApproximation}) 
and question~\ref{qIcebergAndRank}.ii is formulated concerning the following: 
is it possibile to construct a~refined%
\footnote{Iceberg $\iceberg_{n+1}$ refines $\iceberg_n$ if $\cA(\iceberg_n) \subseteq \cA(\iceberg_{n+1})$, 
	where $\cA(\cP)$ is the $\sigma$-algebra generated by partition~$\cP$.} 
sequence of icebergs approximating map~$T$ 
like in the case of rank one transformations?

\subsection{Iceberg} 

Consider an invertible measure preserving transformation $T$ on the standard Lebesgue space $(X,\cA,\mu)$, 
and let $h$ be a positive integer number.

\begin{figure}[th]
  \centering
  \unitlength=1mm
  \includegraphics{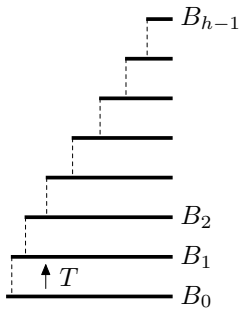} 
  \caption{An icelet} 
  \label{fIcelet}
\end{figure}

\begin{defn}
An {\it icelet\/}%
\footnote{The term {\it icelet\/} is proposed by V.\,Ryzhikov.} 
is a sequence of disjoint sets ${\{B_0,B_1,\ldots,B_{h-1}\}}$ 
such that ${B_{j+1} \subseteq TB_j}$ and ${B_j \in \cA}$. 
\end{defn}

\begin{defn}
A {\it generic iceberg} is a sequence of disjoint measurable sets 
\begin{equation}
	\iceberg = \{ B_{-h+1},\ldots,B_{-1},B_0,B_1,\ldots,B_{h-1} \} 
\end{equation}
such that 
${B_{j+1} \subseteq TB_j}$ for ${j \ge 0}$ and 
${B_{j-1} \subseteq T^{-1}B_j}$ for ${j \le 0}$ (see fig.\,\ref{fdynCAT}). 
Notice that $\iceberg$ is composed of two icelets with common base $B_0$, 
one normal ({\it direct\/}) and one reverse, where {\it reverse icelet\/} is an icelet for~$T^{-1}$. 
We will use notation ${\cup\iceberg = \bigcup_j B_j}$. 
\end{defn}

\begin{figure}[th]
  \centering
  \unitlength=1mm
  \includegraphics{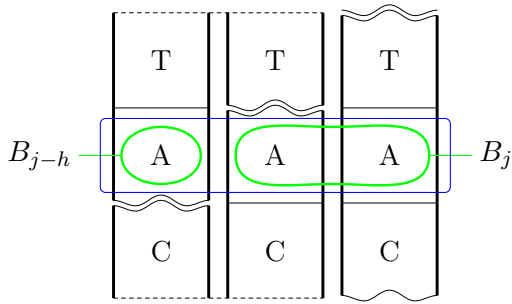} 
  \caption{Partition ${\ice = \{B_j \cup B_{j-h}\}}$ associated with a cyclic iceberg. 
  	The picture is actually drawn on a cylinder, the dashed lines indicate that the top and the bottom of each colums are glued. 
  	These columns of a~cyclic iceberg (we call it {\it fat\/} columns) can be also plot like unsorted set of cycles with 
  	different cut points as shown on fig.\,\ref{fCATRolles}.} 
  \label{fToricCAT}
\end{figure}

\begin{defn}
Let $\iceberg$ be a generic iceberg. 
We say that $\iceberg$ is {\it cyclic\/} if for any point ${x \in B_0}$ 
the total number of iterations towards future and past until leaving the iceberg is equal to $h$, 
i.e.\ ${\# \{ j \in \Set{Z} \where T^j x \in B_j\} \equiv h}$. 
Let us define the {\it cyclic iceberg partition\/} 
\begin{equation}
	\ipart = \{B_j \cup B_{j-h} \where j=0,1,\ldots,h-1\}, 
\end{equation}
which is evidently refined by~$\iceberg$ (see fig.\,\ref{fToricCAT}). 
\end{defn}

\begin{figure}[th]
  \centering
  \unitlength=1mm
  \includegraphics{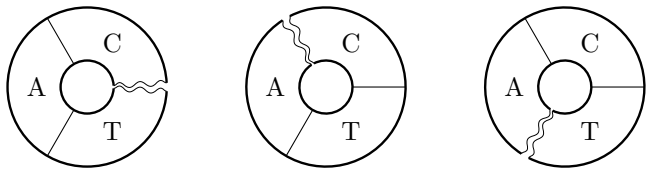} 
  \caption{Unordered set of copies of the word ``$\CAT$'' with different cut points.} 
  \label{fCATRolles}
\end{figure}

\begin{rem}
If iceberg is cyclic, we can redraw fig.\,\ref{fCAT} moving $h$ points up levels with negative indexes 
(in~fact, all icebergs on the illustrations are cyclic). 
After the modification a cyclic iceberg looks like a set of columns 
\begin{equation}
	V_k = \{V_{k,j} \where j \in \Set{Z}_h\}. 
\end{equation}
with different cut points at position ${k = 0,1,\ldots,h-1}$ as shown on fig.\,\ref{fToricCAT}. 
The meaning of cutting will be clear in the forthcoming discussion on ergodic map generated by icebergs. 

A~column before cutting is a sequence of sets $C_{k,j}$ indexed by ${j \in \Set{Z}_h}$ 
with ${\mu(C_{k,j_1}) \equiv \mu(C_{k,j_2})}$ 
such that ${\mu(TC_{k,j} \mid C_{k,j+1}) \approx 1}$
, i.e.\ the set $TC_{k,j}$ is close to~$C_{k,j+1}$. 
When we add cut point at position $k$ this property holds for all $j$ except only one ${j_\times = k}$, 
namely, the case when we pass from level $k$ to level ${k+1}$.

\begin{figure}[th]
  \centering
  \unitlength=1mm
  \begin{picture}(108,52)
      \footnotesize
      \put(0,21)
      {
      	\put(16,0)
      	{
        	\linethickness{1pt}
        	\put(0,0){\line(1,0){8}}
        	\put(0,0){\line(0,1){24}}
        	\put(0,24){\line(1,0){8}}
        	\put(8,0){\line(0,1){24}}
        	\put(2.25,2.4){\large $\syC_3$}
        	\put(2.25,10.4){\large $\syA_3$}
        	\put(2.25,18.4){\large $\syT_3$}
        	\linethickness{0.2pt}
        	\put(0,8){\line(1,0){8}}
        	\put(0,16){\line(1,0){8}}
        }
      	\put(8,-8)
      	{
        	\linethickness{1pt}
        	\put(0,0){\line(1,0){8}}
        	\put(0,0){\line(0,1){24}}
        	\put(0,24){\line(1,0){8}}
        	\put(8,0){\line(0,1){24}}
        	\put(2.25,2.4){\large $\syT_2$}
        	\put(2.25,10.4){\large $\syC_2$}
        	\put(2.25,18.4){\large $\syA_2$}
        	\linethickness{0.2pt}
        	\put(0,8){\line(1,0){8}}
        	\put(0,16){\line(1,0){8}}
        }
      	\put(0,-16)
      	{
        	\linethickness{1pt}
        	\put(0,0){\line(1,0){8}}
        	\put(0,0){\line(0,1){24}}
        	\put(0,24){\line(1,0){8}}
        	\put(8,0){\line(0,1){24}}
        	\put(2.25,2.4){\large $\syA_1$}
        	\put(2.25,10.4){\large $\syT_1$}
        	\put(2.25,18.4){\large $\syC_1$}
        	\linethickness{0.2pt}
        	\put(0,8){\line(1,0){8}}
        	\put(0,16){\line(1,0){8}}
        }
    	}
    	%
    	%
      \put(40,21)
      {
      	\put(16,0)
      	{
        	\linethickness{1pt}
        	\put(0,0){\line(1,0){8}}
        	\put(0,0){\line(0,1){24}}
        	\put(0,24){\line(1,0){8}}
        	\put(8,0){\line(0,1){24}}
        	\put(2.25,2.4){\large $\syA_1$}
        	\put(2.25,10.4){\large $\syT_1$}
        	\put(2.25,18.4){\large $\syC_1$}
        	\linethickness{0.2pt}
        	\put(0,8){\line(1,0){8}}
        	\put(0,16){\line(1,0){8}}
        }
      	\put(8,-8)
      	{
        	\linethickness{1pt}
        	\put(0,0){\line(1,0){8}}
        	\put(0,0){\line(0,1){24}}
        	\put(0,24){\line(1,0){8}}
        	\put(8,0){\line(0,1){24}}
        	\put(2.25,2.4){\large $\syC_3$}
        	\put(2.25,10.4){\large $\syA_3$}
        	\put(2.25,18.4){\large $\syT_3$}
        	\linethickness{0.2pt}
        	\put(0,8){\line(1,0){8}}
        	\put(0,16){\line(1,0){8}}
        }
      	\put(0,-16)
      	{
        	\linethickness{1pt}
        	\put(0,0){\line(1,0){8}}
        	\put(0,0){\line(0,1){24}}
        	\put(0,24){\line(1,0){8}}
        	\put(8,0){\line(0,1){24}}
        	\put(2.25,2.4){\large $\syT_2$}
        	\put(2.25,10.4){\large $\syC_2$}
        	\put(2.25,18.4){\large $\syA_2$}
        	\linethickness{0.2pt}
        	\put(0,8){\line(1,0){8}}
        	\put(0,16){\line(1,0){8}}
        }
    	}
    	%
    	%
      \put(80,21)
      {
      	\put(16,0)
      	{
        	\linethickness{1pt}
        	\put(0,0){\line(1,0){8}}
        	\put(0,0){\line(0,1){24}}
        	\put(0,24){\line(1,0){8}}
        	\put(8,0){\line(0,1){24}}
        	\put(2.25,2.4){\large $\syT_2$}
        	\put(2.25,10.4){\large $\syC_2$}
        	\put(2.25,18.4){\large $\syA_2$}
        	\linethickness{0.2pt}
        	\put(0,8){\line(1,0){8}}
        	\put(0,16){\line(1,0){8}}
        }
      	\put(8,-8)
      	{
        	\linethickness{1pt}
        	\put(0,0){\line(1,0){8}}
        	\put(0,0){\line(0,1){24}}
        	\put(0,24){\line(1,0){8}}
        	\put(8,0){\line(0,1){24}}
        	\put(2.25,2.4){\large $\syA_1$}
        	\put(2.25,10.4){\large $\syT_1$}
        	\put(2.25,18.4){\large $\syC_1$}
        	\linethickness{0.2pt}
        	\put(0,8){\line(1,0){8}}
        	\put(0,16){\line(1,0){8}}
        }
      	\put(0,-16)
      	{
        	\linethickness{1pt}
        	\put(0,0){\line(1,0){8}}
        	\put(0,0){\line(0,1){24}}
        	\put(0,24){\line(1,0){8}}
        	\put(8,0){\line(0,1){24}}
        	\put(2.25,2.4){\large $\syC_3$}
        	\put(2.25,10.4){\large $\syA_3$}
        	\put(2.25,18.4){\large $\syT_3$}
        	\linethickness{0.2pt}
        	\put(0,8){\line(1,0){8}}
        	\put(0,16){\line(1,0){8}}
        }
    	}
  \end{picture}
  \caption{Cyclic rotation of a word indexing levels of the cyclic iceberg leads 
  	to another choice of the base set and the letter defined to be the origin.} 
  \label{fCyclingCAT}
\end{figure}
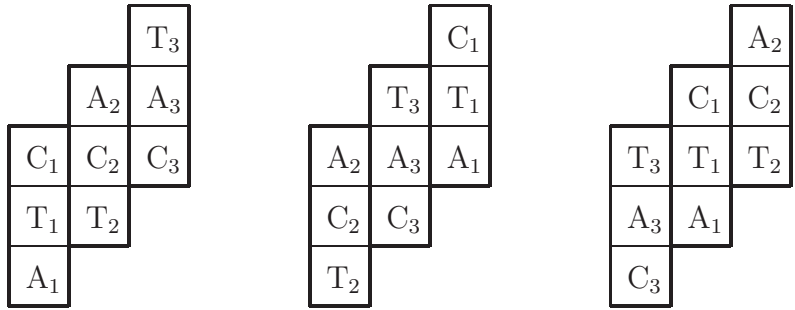

A principal feature of a cyclic iceberg is the posibility of it's cyclic rotation 
which in fact means that we can take any level as the base set of the iceberg (see fig.\,\ref{fCyclingCAT}). 
Clearly a choice of a base set corresponds to a choice of a point in 
the homogeneous space ${\Set{Z}_h}$ playing role of origin. 
Since the acting group $\Set{Z}$ is Abelian this homogeneous space is a group. 
Thus, we have seen that the levels os the iceberg are indexed by elements of the group $\Set{Z}_h$. 
Furthermore, the way we draw an iceberg as a kind of parallelogram is explained 
by the idea to draw the columns in order making the cut points sequence accending, 
and from dynamical point of view an iceberg is just an unordered set of columns 
as shown on fig.\,\ref{fCATRolles}, where each column corresponds to a specific choice of the cut~point. 

What is the meaning of cut point? If we consider a word with glued first and last letters 
as discrete circle $\Set{Z}_h$ then the cut point originates from the boundary of 
a fundamental domain associated with lattice $h\Set{Z}$.  
\end{rem}

\subsection{Iceberg approximation as dynamical invariant} 

In the sequel for simplicity we will use the term {\it iceberg\/} 
as shorter equivalent of the term {\it cyclic iceberg\/}. 
Actually the notion of {\it cyclic iceberg\/} is connected with the concept of 
iceberg approximation explained in the next definition, and defined using rotation operator~$\rho_\a$. 

\begin{defn}\label{defIcebergApproximation} 
We say that a measure preserving map $T$ admits {\it iceberg approximation\/} if 
given a finite measurable partition associated with an alphabet $\Set{A}$ 
for any ${\eps > 0}$ there exists a word $W_\eps$ in the alphabet $\Set{A}$ such that 
for ${(1-\eps)}$-fraction of orbits $(x_n)$ the subword of length $N(\eps)$ in $(x_n)$ 
starting from $x_0$ is $\eps$-covered by rotations $\rho_\a(W_\eps)$ of the word~$W_\a$. 
\end{defn}

\begin{thm}
Iceberg approximation property is a dynamical invariant. 
\end{thm}

\begin{figure}[th]
  \centering
  \unitlength=1mm
  \includegraphics{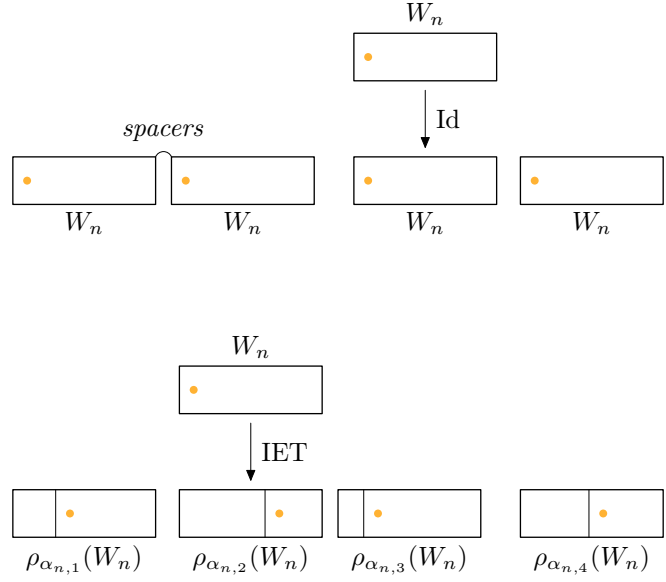} 
  \caption{Comparing rank one and iceberg approximation} 
  \label{fIcebergAndRank}
\end{figure}

A question on finding an invariant property for iceberg transformations is proposed by B.\,Weiss. 
To~see the difference between rank one and iceberg approximation let us draw 
compare orbit structure for some rank one and some iceberg transformation, see fig.\,\ref{fIcebergAndRank}. 
%

\begin{proof} 
If two transformations $T$ and $S$ are equivalent via a measure preserving map $I$ then 
we can map a finite partition $\cP$ for one map to a finite partition $I(\cP)$ for another,
and if the $\cP$-coding of $T$-orbits is covered by words $\rho_{a_j}(W)$ the same property 
holds for $I(\cP)$-coding of $S$-orbits.
\end{proof}

\subsection{Construction of iceberg transformation} 

%
%

\begin{defn}
An {\it iceberg transformation\/} ({\it without spacers\/}) 
is approximated by a sequence of icebergs $\ice_n$ of height $h_n$ 
such that any next iceberg in the sequence refines the privious one 
according to the rotated words concatenation 
procedure defined in section~\ref{sIcebergAtAGlance}. 
%
Let us fix heights $h_n$ and rotations $\a_{n,y}$, and suppose that  
\begin{equation}
	h_{n+1} = q_n h_n, \quad h_0 = 1, \quad q_n \in \Set{Z}, \quad q_n \ge 2. 
\end{equation}
Set ${X_n = \{0,1,\ldots,h_n-1\}}$, and consider projections 
\begin{equation}
    \phi_n \Maps X_{n+1} \to X_n, \qquad n = 0,1,\ldots 
\end{equation}
where 
\begin{equation}
    \phi_n(yh_n+t) = t - \a_{n,y} \pmod{h_n}, \qquad 
    0 \le t < h_n, \quad y = 0,1,\ldots,q_n-1. 
\end{equation}
Restricted to any interval $[yh_n,(y+1)h_n)$ the map $\phi_n|_{[yh_n,(y+1)h_n)}$ is exactly the inverse map 
to the rotation operator~$\rho_{\a_{n,y}}$ in concordance with~\eqref{eDefWnPlusOne}. 
It is clear that ${\mu_{n+1}(\phi_n^{-1}A) = \mu_n(A)}$ for any set~$A$. 
Thus, we can endow the inverse limit $X$ of spaces $(X_n,\mu_n)$ with Borel measure $\mu$ such that 
${\mu(\{x \where x_n = t\}) \equiv h_n^{-1}}$, where 
\begin{equation}
	X = \{x = (x_0,x_1,\ldots,x_n,\ldots) \where \phi_n(x_{n+1}) = x_n\}. 
\end{equation}
Clearly, the $n$-th coordinate $x_n$ satisfies condition 
\begin{equation}\label{ePlusOne}
	\phi_{n-1}(x_n+1) = \phi_{n-1}(x_n) + 1 
\end{equation}
with $\mu$-probality $1/h_{n-1}$ (the probability to encounter the cut point). 
Since ${h_n \to \infty}$ condition~\eqref{ePlusOne} holds for coordinates $x_{n_0},x_{n_0+1},\ldots$ 
starting from some index~$n_0$ with $\mu$-probability~$1$. We call these points $x$ {\it regular}. 
Let us define the {iceberg transformation\/} map $T$ for a regular point $x$ by formula 
\begin{equation}\label{eDefT}
	T \Maps (x_0,\ldots,x_{n_0-1},x_{n_0},x_{n_0+1},\ldots) \mapsto (x_0^{++},\ldots,x_{n_0-1}^{++},x_{n_0}+1,x_{n_0+1}+1,\ldots), 
\end{equation}
where the head of the sequence is recovered in the only way such that $Tx$ becomes a correct sequence 
with the property ${\phi_n((Tx)_{n+1}) = (Tx)_n}$, namely we set 
\begin{equation} 
	x_{n_0} + 1 				\;\stackrel{\phi_{n_0-1}}\mapsto \; 
	x_{n_0-1}^{++} 			\;\stackrel{\phi_{n_0-2}}\mapsto \; 
	x_{n_0-2}^{++} 			\;\stackrel{\phi_{n_0-3}}\mapsto \; 
	\ldots 							\;\stackrel{\phi_1}\mapsto \; 
	x_1^{++} 						\;\stackrel{\phi_0}\mapsto \; 
	x_0^{++}. 
\end{equation}
\end{defn}

\begin{lem}
The maps $T$ is a measure preserving invertible transformation of the space $(X,\mu)$. 
\end{lem}

\begin{proof}
On the $n$-th level the map $T$ is close to the rotation map ${t \mapsto t+1}$ 
which preserves $\mu_n$, and we can take arbitrary~$n$, hence, $T$ is measure preserving. 

Invertibility of~$T$ is a corollary of an important observation that 
$T$ is generated by the underlying transformation ${t \mapsto t+1}$ 
of the acting group~$\Set{Z}$ having ${t \mapsto t-1}$ as inverse. 
Indeed, if look at $n$-th level the transformation ${t \mapsto t+1 \pmod{h_n}}$ 
of the homogeneous space $X_n$ associated with $n$-th iceberg is induced by 
the shift transformation ${t \mapsto t+1}$ of~$Z$. 
Thus, the inverse map for $T$ is defined in the same way as~$T$ but using the map ${t \mapsto t-1}$. 

Walking by $T$ and $T^{-1}$ 
it is interesting to observe the reversibility of jumps 
\begin{equation}\label{eTwoJumps}
	T \Maps x_n \mapsto x_n^{++} \quad \text{and} \quad T^{-1} \Maps \tilde x_n \mapsto \tilde x_n^{--}, 
\end{equation} 
where ${\tilde x_n = x_n^{++}}$. 
Under a {\it jump\/} we understand the case when ${x_n^{++} \not= x_n + 1}$. 
In fact, crossing the cut point ($\times$) after the (first) letter ``$\syC$'' 
in a join ${\ldots\ATC_\times \TCA\ldots}$ we jump to the position of letter ``$\syT$'' 
instead of passing to the next letter ``$\syA$'' in the original order ${\syC \mapsto \syA \mapsto \syT \mapsto \syC}$. 
\end{proof}

\subsection{Interpretation of dynamics}

\begin{figure}[th]
  \centering
  \unitlength=1mm
  \includegraphics{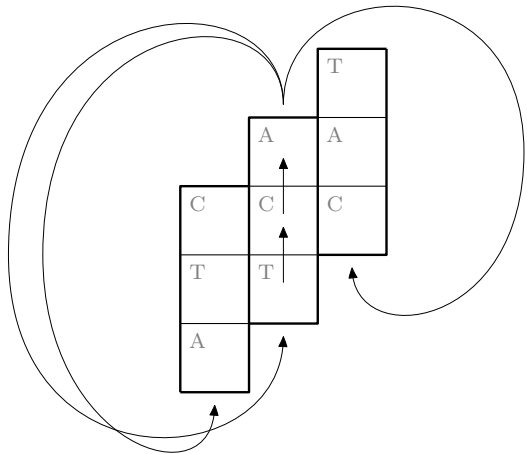} 
  \caption{Jumps under the action of Poincar\'e map} 
  \label{fJumpingCAT}
\end{figure}

Let us draw on one picture {\em two\/} steps of the iceberg map construction 
to understand clearly the dynamics of~$T$. We are going to start with 
the picture shown at fig.\,\ref{fCAT} and to gradually add details on this picture. 
The first step is to interpret the Poincar\'e map $\tilde T_{E_n}$. 

\begin{figure}[th]
  \centering
  \unitlength=1mm
  \includegraphics{iceart5.eps} 
  \caption{} 
  \label{fJumpsOnXn}
\end{figure}

\begin{defn}
Let us call {\it fat columns} the vertical columns $V_k$ on figures \ref{fCAT}, \ref{fdynCAT} and \ref{fJumpingCAT}. 
Intersecting fat column $V_k$ with levels we get partition 
\begin{equation}
	V_k = V_{k,-h+k} \sqcup V_{k,-h+k+1} \sqcup \ldots \sqcup V_{k,k-1}, \qquad V_{k,j} = V_k \cap B_j, 
\end{equation}
where ${k = 1,2,\ldots,h}$. Observe that the number of all fat columns is~$h_n$. 
Further, let us use the term {\it thin column\/} for a vertical subcolumn of a fact column 
corresponding to one rotated copy of the word~$W_n$ inside the word~$W_{n+1}$ 
(see fig.\,\ref{fCATDetailed}). In other wordth, a thin column corresponds to 
an interval ${ [yh_n,(y+1)h_{n+1}) }$ in~$X_{n+1}$. There are totally $q_n$ thin columns. 
\end{defn}

\begin{figure}[th]
  \centering
  \unitlength=1mm
  \includegraphics{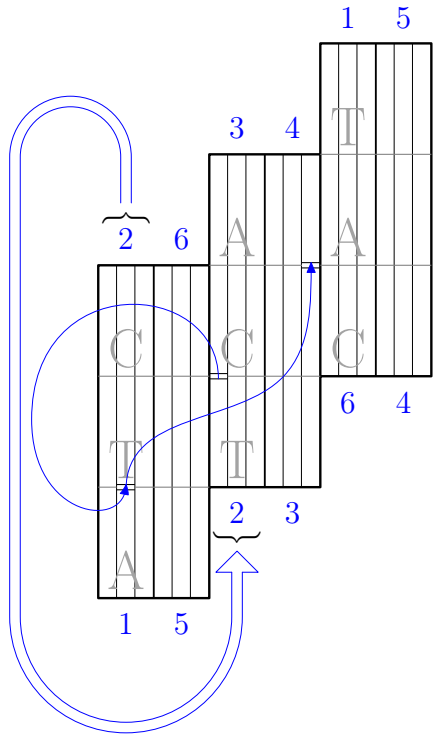} 
  \caption{Iceberg transformation dynamics.} 
  \label{fCATDetailed}
\end{figure}

\begin{figure}[th]
  \centering
  \unitlength=1mm
  \includegraphics{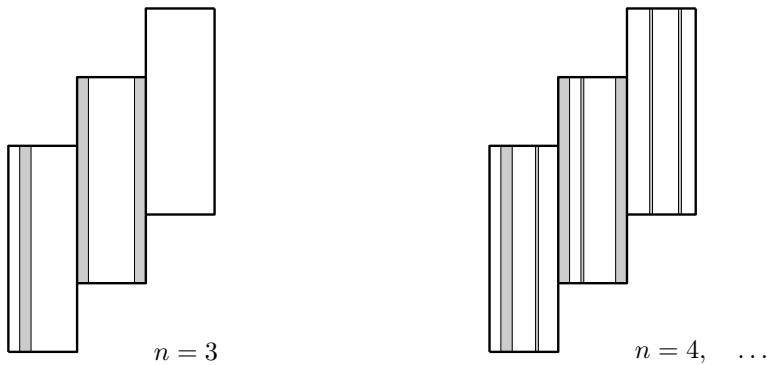} 
  \caption{Grayed colomns contain jumps. 
  	If a point ${x \in B_j}$ and if $x$ is located in 
  	the white area ({\it body\/}) then ${Tx \in B_{j+1}}$.} 
  \label{fBodyOfCAT}
\end{figure}

Clearly the meaning of Poincar\'ee map is to express the connection of two thin columns 
corresponding to a pair of adjacent rotated copies $\rho_{n,y}(W_n) \rho_{n,y+1}(W_n)$ 
which are subwords of~$W_{n+1}$. For example, if this pair is $\CAT.\ATC$ the Poincar\'ee map 
sends the top set in a thin column included in fat column $V_3$~(``$\CAT$'') to the bottom set of some 
thin column in fat column $V_1$ (``$\ATC$''). From the point of view of observer watching the coordinate~$x_n$
in the homogeneous space $\Set{Z}_{h_n}$ the following occurs: with probability ${1-\eps}$ point $x_n$ 
moves one step forward ${x_n \mapsto x_n+1}$, and with small probability $\eps$ it {\em jumps\/} to 
any other point in~$\Set{Z}_{h_n}$. At this level we cannot say precisely how this jumpings are 
distributed (this is another definition of iceberg, cyclicity means that forward and backward tracks 
before jumping have common length~$h_n$). In terms of iceberg we can explain this behavior as follows. 
A~point $x$ under~$T$ moves up with probability $(1-\eps_2)$, and reaching the top level in a fat column 
it make a random jump to some bottom set in another fat column (see fig.\,\ref{fJumpsOnXn}). 
In other words we are going to show geometrically 
what happens when we rotate and concatenate copies of word ${W_2 = \CAT.\ATC.\TCA.\TCA.\CAT.\ATC}$. 
Suppose that ${q_2 = 3}$ and $W_3$ is the result of concatenation of 
the following three rotated copies 
\begin{align}
	{\mathrm{ CATATCT_\times CATCACATATC }} \\
	{\mathrm{ CATA_\times TCTCATCACATATC }} \notag \\
	{\mathrm{ CATATCTCATC_\times ACATATC }} \notag 
\end{align}
namely, 
\begin{equation}
	W_3 = {\scriptstyle{\mathrm{ 
		CATCACATATC \,|\, CATATCT \:\seppoint\: 
		TCTCATCACATATC \,|\, CATA \:\seppoint\: 
		ACATATC \,|\, CATATCTCATC
		}}}
\end{equation}
At figure~\ref{fCATDetailed} both first and second Poincar\'e maps are shown, 
and we can continue this procedure. Let us mark the columns containing jumps as grayed 
(see fig.\,\ref{fBodyOfCAT}).

\subsection{Approximation and rank}

\begin{lem}\label{lemIcebergApprox}
Let $T$ be an iceberg transformation. 
A function $f \in L^2(X,\mu)$ can be approximated by a sequence of $\ice_n$-measurable functions. 
In other words, $T$ admits iceberg approximation. 
In particular, function $f$ can be approximated by a sequence of functions $f_n$, where 
$f_n$ is constant on levels of icebergs~$\iceberg_n$. 
\end{lem}

\begin{rem}
Recall that $\iceberg_n$-measurable functions are constant 
on levels ${\{B_{-h+1},\ldots,B_0,\ldots,B_{h-1}\}}$ of iceberg~$\iceberg_n$, 
and $\ice_n$-measurable functions are constant on levels $B_j$ 
and in addition it take same value both on $B_{j-h}$ and~$B_j$. 
Simply speaking, different values of~$f$ correspond to different letters on 
fig.\ \ref{fCAT}. 
\end{rem}

\begin{lem}\label{lemIcebergStructure}
Suppose that ${q_n \to \infty}$. 
There exist sub-icebergs $\icebody_n$ for an icegerg transformation $T$ such that 
$T$ lifts any point ${x \in \icebody_n}$ exactly to the upper level, and 
${\mu(\cup\icebody_n \mid \cup\ice_n) \to 1}$ as ${n \to \infty}$. 
In other words, there exist sets ${V^\ib_{k,j} \subseteq V_{k,j}}$ such that   
\begin{equation}
	TV^\ib_{k,j} = V^\ib_{n,k,j+1}, \qquad j = -h+k,\ldots,k-1, 
\end{equation}
and ${\mu(\cup\icebody_n \mid \cup\ice_n) \to 1}$, where 
\begin{equation}
	\cup\icebody_n = \bigcup_{k=1}^{h_n} \:\bigcup_{j=-h+k}^{k-1} V^\ib_{n,k,j}. 
\end{equation}
\end{lem}

It can be easily seen from the condition that each set $\icebody_n$ (we call it {\it body\/}) 
can be choosen as the maximal measurable set with the property 
{\it point make no jumps until reaching the top set in a column}. 

\begin{proof}
Consider the thin vertical subcolumns corresponding to the rotated copies of $W_n$ 
in the word $W_{n+r}$. The are totally ${Q_{n,n+r} = q_n q_{n+1}\cdots q_{n+r-1}}$ such columns. 
Let us watch the number $Q^\ib_{n+r}$ of thin subcolumns at level ${n+r}$ 
which is not touched by cut points. 
For example, ${Q_{n,n+1} = Q^\ib_{n,n+1} = q_n}$ and 
\begin{equation}
	Q_{n,n+2} = q_n q_{n+1}, \qquad 
	Q^\ib_{n,n+2} \ge Q_{n,n+2} - q_{n+1}, 
\end{equation}
since passing from level $n+1$ to level $n+2$ the word $W_{n+2}$ is a sequence of rotated copies of~$W_{n+1}$, 
and potentially we can see $q_{n+1}$ cut point touching small words originated from~$W_n$. More precisely, 
passing from level ${n+r}$ to ${n+r+1}$ we can get one additional cutting inside 
one thin subcolumn at level ${n+r}$, i.e.\ 
\begin{equation}
	Q^\ib_{n,n+r+1} \ge Q^\ib_{n,n+r} q_{n+r} - q_{n+r} = (Q^\ib_{n,n+r} - 1) q_{n+r}. 
\end{equation}
Let us denote as $A_{n,n+r}$ the union of all thin subcolumns build at level ${n+r}$, 
\begin{equation}\label{lemIcebergStrEstA}
	\mu(A_{n,n+r+1} \mid \cup\ice_n) = \frac{Q^\ib_{n,n+r+1}}{Q_{n,n+r+1}} = 
	\frac{Q^\ib_{n,n+r+1}/Q^\ib_{n,n+r}}{q_{n+r}} \frac{Q^\ib_{n,n+r}}{Q_{n,n+r}} \ge 
	\left( 1 - \frac1{Q^\ib_{n,n+r}} \right) \, \mu(A_{n,n+r} \mid \cup\ice_n). 
\end{equation}
Since $q_n \to \infty$ we can assume that ${q_{n+r} \ge 3}$, hence, 
${Q_{n,n+r} \to \infty}$ and ${Q_{n,n+r}-1 \ge \frac12 Q_{n,n+r}}$. Therefore 
\begin{equation}
	Q^\ib_{n,n+r+1} \ge Q^\ib_{n,n+r} \cdot \frac{q_{n+r}}2 \ge Q^\ib_{n,n+r} \cdot \frac32 \ge \left( \frac32 \right)^r, 
\end{equation}
and lemma follows from estimate~\eqref{lemIcebergStrEstA}. 
\end{proof}

\begin{rem}
Let us note that existence of body $\icebody_n$ such that 
${\mu(\cup\icebody_n \mid \cup\iceberg_n) \to 1}$ is stronger than 
correctness of the definition of~$T$. 
%
The method used to prove lemma~\ref{lemIcebergStructure} 
is purely combintorial and requires ${q_n \to 0}$. 
At the same time, applying ergodic theorem 
(we only need condition ${\sum_n h_n^{-1} < \infty}$) 
it is not hard to see that this requirement 
can be omitted, but we not put this proof here.  
\end{rem}

Now we turn to estimation of spectral multiplicity of~$T$. 

\begin{defn}\label{defLocalRankPartition}
Given a Rokhlin tower $\tow$ 
let us use the notation $\twp^\eps$ for the measurable partition into levels of the tower~$\tow$ 
and distinct points outside $\cup\tow$. Notice that ${\twp \preceq \twp^\eps}$. 
\end{defn}

\begin{defn}\label{defLocalRank}
{\it Local rank\/} $\beta(T)$ of a measure preserving transformation $T$ 
is the maximal value $\beta$ with the following property: there exists 
a sequence of towers $\tow_n$ such that ${\mu(\cup\tow_n) \to \beta}$ as ${n \to \infty}$ 
and for any measurable set~$A$ there exists $\twp^\eps_n$-measurable sets $A_n$ 
with ${\mu(A_n \syms A) \to 0}$ as ${n \to \infty}$. 
\end{defn}

\begin{figure}[th]
  \centering
  \unitlength=1mm
  \includegraphics{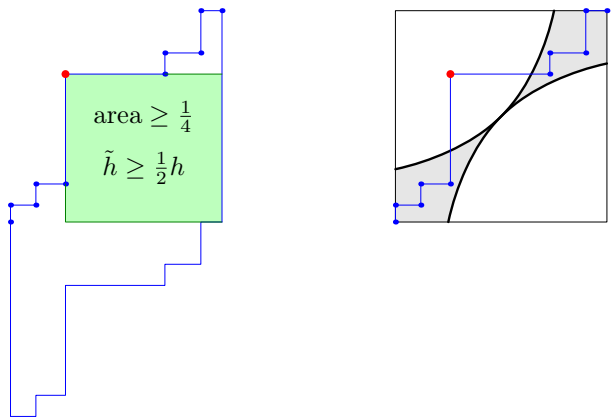} 
  \caption{Proving $1/4$-local rank property. 
  	One vertex on the path visits the white area inside the unit square. 
  	This vertex corresponds to a subtower of the iceberg shown as (green) rectangular area greater than~$1/4$. 
  	On the left picture an iceberg is shown corresponding to the path in the unit square.} 
  \label{fLocalRankProof}
\end{figure}

The following theorem is an immediate geometrical corollary of lemma~\ref{lemIcebergStructure}. 

\begin{thm}\label{thmIcebergLocalRank}
If $T$ is an iceberg transformation with ${q_n \to \infty}$ then $T$ is of $1/4$-local rank. 
\end{thm}

\begin{rem}
V.\,Ryzhikov observed \cite{VRyzhLocalRank14} that $\beta(T\times T) = 1/4$ for a typical $T$. 
\end{rem}

\begin{rem}
In the proof we essentially use the fact that $T$ is approximated by {\em cyclic\/} icebergs. 
\end{rem}

\begin{proof}
The idea of the proof is based on the following observation: 
since $\icebody_n$ approximates $\iceberg_n$ then we can use $\icebody_n$ in investigation of dynamcis of~$T$. 
The $1/4$-local rank property can be proved using the following illustration. 
Consider a pair $(x,y)$, where ${0 \le x,y \le 1}$, and a subset of the unit square (see fig.\,\ref{fLocalRankProof}) 
\begin{equation}
	G = \Bigl\{ 
		(x,y) \in [0,1]\times [0,1] \where 
		(1-x)y \ge \frac12 \quad \text{or} \quad x(1-y) \ge \frac12 
	\Bigr\}. 
\end{equation}
Now we play the following geometric game. Starting from $(0,0)$ we have to reach $(1,1)$ 
moving continuously along horizontal or vertical lines (right or up). 
Evidently, one can find a turn point $p$ on the trajectory such that ${p \in G}$. 
Moreover, we can also choose such a point $p$ with the following additional property: 
after the next step it will fall into the upper grayed area. 
This point $p$ describes a subtower $\tow_n$ of $\icebody_n$ such that 
${\mu(\cup\tow_n) \ge \frac14 \mu(\cup\icebody_n)}$. 
The second requirement to the point~$p$ proveides a subtower of height ${\tilde h_n \ge h_n}$. 
Finally, recall that ${\mu(\cup\icebody_n) \to 1}$. 

So, a sequence of Rokhlin subtowers $\tow_n$ is found 
such that levels of $\tow_n$ are included into levels of~$\iceberg_n$. 
We know that for any measurable set~$A$ there exist sets $A_n$ which is a union of levels of~$\iceberg_n$ 
such that ${\mu(A_n \syms A) \to 0}$ as ${n \to \infty}$. Hence, the intersection 
of the set $A_n$ with the tower $\cup\tow_n$ is a union of levels of this tower. 
\end{proof}

\begin{rem}
This corollary remains true for $T$ having property of iceberg approximation. 
Actually the requirement ${\sum_n h_n^{-1} < \infty}$ is needed 
to achive certainly the property ${\beta(T) \ge 1/4}$. 
\end{rem}

\begin{thm}
An automorphism $T$ with the property of iceberg approximation has $1/4$-local rank. 
\end{thm}

\begin{figure}[th]
  \centering
  \unitlength=1mm
  \includegraphics{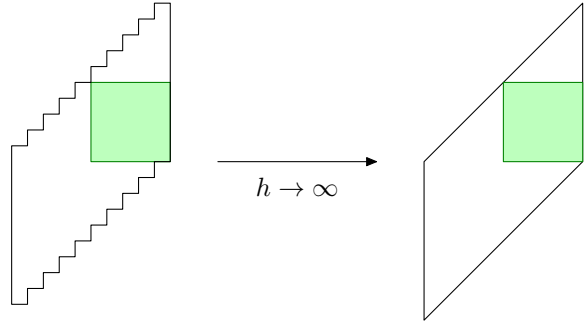} 
  \caption{The area of maximal rectangle fit into the parallelogram is close to~$1/4$. 
  	Green rectangle corresponds to a Rokhlin subtower of the iceberg.} 
  \label{fLocalRankSimpleDemo}
\end{figure}

\begin{rem}
The proof of the above corollary is not so evident. So, to give a simple explanation of the effect 
let us consider an iceberg with large number of fat columns having measure close to $1/h_n$. 
In this case the iceberg looks like parallelogram (see fig.\,\ref{fLocalRankSimpleDemo}) 
which is composed of two triangles with verticies $(0,0)$, $(0,h)$, $(-h,0)$, $(-h,-h)$. 
This illustration will be often used in the sequel. 
\end{rem}

\begin{defn}\label{defUniformIceberg}
Let us call the sequence of icebergs $\iceberg_n$ {\it uniform\/} 
if the vector 
$(\mu(V_{n,1}),\ldots,\mu(V_{n,h_n}))$ is $l^1$-close to the uniform distribution $(h_n^{-1},\ldots,h_n^{-1})$. 
\end{defn}

\begin{exmp}
We say that an iceberg transformation $T$ has {\it Morse property\/} 
if ${h_{n-1} \mid \a_{n,y}}$ 
for any ${y=0,\ldots,q_n-1}$. Remark that Morse property implies 
${\cup\icebody_n = \cup\iceberg_n}$. 
To give an example of Morse iceberg system, suppose that for any~$n$ 
the word $W_n$ is splitted into a~pair of equal blocks: ${W_n = A_n B_n}$, 
${|A_n| = |B_n| = h_{n-1}}$, 
then by definition 
we set
\begin{equation}
	W_{n+1} = A_n B_n B_n A_n. 
\end{equation}
Clearly, starting with $W_0 = 01$ we get the classical Morse sequence. 
Further, if ${W_n = A_n B_n C_n}$, we can consider Morse iceberg map $T_{(3)}$ of order~$3$: 
\begin{equation}
	W_{n+1} = A_n B_n C_n \:.\: B_n C_n A_n \:.\: C_n A_n B_n 
\end{equation}
and so on. Observe that theorem~\ref{thmIcebergLocalRank} can be applied 
to iceberg transformations with Morse property and we see that 
\begin{equation}
	\beta(T_{(r)}) = \left( \frac{r+1}{2r} \right)^2 \qquad \text{if} \quad r=2k+1
\end{equation}
and
\begin{equation}
	\beta(T_{(r)}) = \frac{r(r+2)}{4r^2} \qquad \text{if} \quad r=2k.
\end{equation}
\end{exmp}

\begin{defn}
We say that a measure preserving map $T$ 
has iceberg rank $r$ if there exists a sequence of partitions 
${\cP_n = \ice_n^{(1)} \vee \ldots \vee \ice_n^{(r)}}$ such that 
any measurable set $A$ is approximated by $\cP_n$-measurable sets $A_n$. 
Let us denote $r_I(T)$ the minimal $r$ with this property. 
\end{defn}

\begin{defn}
Let us define $\beta_I(T)$ to be the maximal value $\beta$ such that 
a sequence of icebergs $\ice_n$ approximates $\sigma$-algebra and 
${\mu(\cup\iceberg_n) \to \beta}$ as ${n \to \infty}$. 
\end{defn}

\subsection{Iceberg transformations with spacers}

Let us improve the definition adding spacers between rotated copies of word~$W_n$:
\begin{equation}
	W_{n+1} = 	\rho_{n,\a_{n,0}}(W_n) 1^{s_{n,0}} 
							\rho_{n,\a_{n,1}}(W_n) 1^{s_{n,1}} 
							\ldots
							\rho_{n,\a_{n,q_n-1}}(W_n) 1^{s_{n,q_n-1}}. 
\end{equation}
If we require condition \eqref{eFiniteMeasRankOne} this procedure 
generates a measure preserving transformation on a Lebesgue space. 

\begin{rem}
There is one special case of iceberg map with spacers: 
we do not put spacers between adaject words 
$\rho_{n,\a_{n,0}}(W_n) 1^{s_{n,0}} \rho_{n,\a_{n,1}}(W_n) 1^{s_{n,1}}$ 
but we add only one group of spacers $1^{s_{n,q_n-1}}$ to the tail of~$W_{n+1}$. 
This method helps to provide $h_n$ with special arithmetic properties. 
\end{rem}

\begin{quest}\label{qIcebergAndRank}
Is the following true or not?
\begin{itemize}
	\item [(i)] There exists a map with iceberg approximation but not rank one? 
	\item [(ii)] Any map with iceberg approximation is isomorphic to a map built using 
		iceberg construction with spacers? In other word, is it always possible to refine 
		the sequence of icebergs for a map with iceberg approximation? 
	\item[(iii)] The classes of iceberg maps with and without spacers are not identical?
	\item[(iv)] There exists a map with iceberg approximation which is of infinite rank?
	\item[(v)] There exists a map $T$ with iceberg approximation such that ${\hat T^{j_k} \to \frac12(\Id + \hat T)}$?
	\item[(vi)] What are $r_I(T)$ and $\beta_I(T)$ for a typical $T$?
	\item[(vii)] What is the vaue of ${\beta_I(T \times T)}$ for a typical $T$? 
	\item[(viii)] For any map $T$ with iceberg approximation ${T \times T}$ is of local rank one? 
	\item[(ix)] The entropy $h(T) > 0$ if ${\beta_I(T) > 0}$? 
	\item[(x)] Given a transformation $T$ with iceberg approximation is it true that 
		symbolic complexity of~$T$ is always sub-exponential (see~\cite{Ferenczi2})?
\end{itemize}
\end{quest}


\section{Simplicity of spectrum}\label{sSimpleSp}

We know by theorem~\ref{thmIcebergLocalRank} that iceberg approximation is stronger than 
$1/4$-local rank approximation, therefore, spectral multiplicity of an iceberg transformation 
${\spMult(T) \le 4}$ (see~corollary~\ref{corIcebergSpMult}). 
At the same time, surprisingly, iceberg approximation implies simplicity of spectrum 
for a wide class of transformations. 
We are going to prove this for a class of maps with randomized jumps and for a class of maps 
with jumps given by a pseudo-random substitution on the set of columns. 


\begin{lem}\label{lemSpMultEst} 
Let $U$ be a unitary operator in a separable Hilbert space $H$, 
let $\sigma$ be the maximal spectral type measure, and 
let $\Mult(z)$ be the multiplicity function of the operator~$U$. 
If ${\Mult(z) \ge m}$ on a set of positive $\sigma$-measure than 
one can find $m$ orthogonal elements of unit length $f_1,\dots,f_m$ such that 
for any cyclic subspace ${Z \subset H}$ (with respect to~$U$) and for any 
$m$ elements $g_1,\dots,g_m \in Z$, of equal length ${\|g_i\| \equiv a}$ 
the following is true 
\begin{equation}\label{eSpMultEst} 
    \sum_{i=1}^m \|f_i - g_i\|^2 \ge m (1 + a^2 - 2a/\sqrt{m}). 
\end{equation}
\end{lem}

\begin{cor}
If $T$ is a $\beta$-local rank transformation then ${\spMult(T) \le 1/\beta}$. 
In particular, any rank one transformation have simple spectrum. 
\end{cor}

\begin{cor}\label{corIcebergSpMult}
Since any transformation $T$ with iceberg approximation 
has $1/4$-local rank property, ${\spMult(T) \le 4}$. 
Further, for any map $T$ with $\beta$-icerberg approximation property, ${\spMult(T) \le [4/\beta]}$. 
\end{cor}

The discussion throughout this section concerns simplicity of spectrum for iceberg transformation. 

\subsection{Preliminary calculations} 


Suppose that the maximal spectral multiplicity ${m(T) \ge 2}$. 
So, applying lemma~\ref{lemSpMultEst} we see that two functions $f_1$ and $f_2$ should exist 
suth that any cyclic subspace $Z$ contains elements $g_1, g_2$ with the propertiy ${\|g_1\| = \|g_2\| = a}$ 
and satisfying 
\begin{equation}
    \|f_1 - g_1\|^2 + \|f_2 - g_2\|^2 \ge 2(1 + a^2 - 2a/\sqrt{2}).
\end{equation}
We will show that for a class of transformations 
there exist a cyclic subspace approximating both $f_1$ and $f_2$, and the contradiction will follow. 
Functions $f_i$ can be approximatied by $\ice_{n_0}$-measurable functions, and without loss of generality 
we can assume that $f_1$ are $f_2$ are $\ice_{n_0}$-measurable. Since the arguments concerning 
approximation precision are the same both for $f_1$ and~$f_2$, we will work wiht one function $f_1$. 
Recall that $f_1$ is $\ice_n$-measurable for any ${n \ge n_0}$. Let us denote $f_{(n)}$ the lifting 
of the function $f_1$ to the iceberg $\ice_n$. 
Let ${b_n = {\bf 1}_{B_{n,0}}}$ be the indicator of the base set $B_{n,0}$ of the iceberg~$\ice_n$. 
We~have 
\begin{equation}\label{capprf}
    f = \sum_{j=0}^{h_n-1} f_{(n)}(j)\, S^j b_n 
    	= \sum_{j=-(h_n-1)/2}^{(h_n-1)/2} f_{(n)}(j)\, S^j b_n, \qquad 
    n \ge n_0, 
\end{equation}
where by definition ${S B_j = B_{j+1}}$ is the operator in $L^2(X_n)$ corresponding to the rotation ${t \mapsto t+1}$.

\begin{figure}[th]
  \centering
  \unitlength=1mm
  \includegraphics{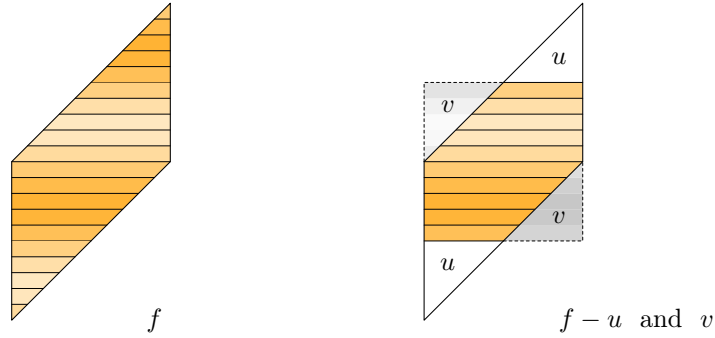} 
  \caption{Proving simplicity of spectrum} 
  \label{fSimpleSpProof}
\end{figure}

{\it Idea}. The following way of proving simplity of spectrum one can call ``$3/4$-strategy''. 
We approximate function $f$ by the iterations $T^j b_n$ of the indicatior of the base level $B_{n,0}$, 
where ${j=-(h_n-1)/2,\ldots,(h_n-1)/2}$, 
on the subset of the phase space $X$ of measure $\approx 3/4$ (see fig.\,\ref{fSimpleSpProof}). 
Assume for simplicity that ${h_n = 2\Set{Z}+1}$. 
Remark that the estimates below used by this approach is 
not a~priory necessary for simplicity. 

Without loss of generality we can assume that ${\int f \,d\mu = 0}$. 
Consider sets $G_n$ and $E_n$ of measure $\approx 3/4$ and $\approx 1/4$ respectively, where 
$G_n$ corresponds to the central area on figure~\ref{fSimpleSpProof} and $E_n$ is the union 
of two triangular areas remote from the base level, 
\begin{equation}
	\cup\iceberg_n = G_n \cup E_n, \qquad 
	G_n = \bigcup_{j = -(h_n-1)/2}^{(h_n-1)/2} B_{n,j}, \qquad 
	E_n = \bigcup_{j = -h_n+1}^{-(h_n-1)/2-1} B_{n,j} \cup 
				\bigcup_{j = (h_n-1)/2+1}^{h_n-1} B_{n,j}
\end{equation}
The function $f$ is approximated by the function 
\begin{equation}
	g = \sum_{j=-(h_n-1)/2}^{(h_n-1)/2} f_{(n)}(j)\, T^j b_n
\end{equation}
and $g$ can be represented in the following way: 
\begin{equation}
	g = f-u + v, \qquad f-u = f|_{G_n}, \quad u = f_{E_n}, 
\end{equation}
and $v$ is uniquely definied from the equation. The meaning of the function $v$ can be explained 
as follows. Walking from the base of the iceberg under action of~$T$ a point $x$ is moving 
in the vertical direction, and the value $f(T^jx)$ is recovered by the index of level, 
namely, ${f(T^jx) = f|_{B_j}}$. Though, when $x$ approaches to the top set of a fat column 
it makes jump to the bottom set of some column and continue vertical motion in that column. 
So, the function $v$ collects the parts of the iterates $T^j b_n$ leaving the set~$G_n$. 
More precisely, we can write 
\begin{equation}
	T^j b_n = {\bf 1}_{B_{n,j}} + \xi_j, \qquad v = 
		\sum_{j=-(h_n-1)/2}^{(h_n-1)/2} f_{(n)}(j)\, \xi_j, 
\end{equation}
and
\begin{equation}
	u = \sum_{j=-h_n+1}^{-(h_n-1)/2-1} f_{(n)}(j)\, {\bf 1}_{B_j} + 
			\sum_{j=(h_n-1)/2+1}^{h_n-1} f_{(n)}(j)\, {\bf 1}_{B_j}. 
\end{equation}
If possible we omit index $n$ for simplicity.

\begin{lem}\label{commonLemSimplicity}
Suppose that ${\scpr<u,v> \to 0}$ and ${\scpr<f,v> \to 0}$ as ${n \to \infty}$ then 
asymptotically ${ \|f-g\|^2 \to 1/2 }$ and ${ \|g\| = (1+o(1))\|f\| }$. 
In particular, $T$ has simple spectrum. 
\end{lem}

\begin{proof}
Assume that $f$ has zero mean and ${\|f\| = 1}$. 
Let us notice that $v(x) = u(\Phi x)$ for some measure preserving invertible map~$\Phi$, 
hence, ${\|v\|^2 = \|u\|^2}$. Further, we have 
\begin{multline}
    a^2 = \|g\|^2 = \|f - u + v\|^2 = 
    \|f\|^2 + \|u\|^2 + \|v\|^2 - 2\Re\scpr<f,u> + 2\Re\scpr<f,v> - 2\Re\scpr<u,v> \approx \\
    \approx
    \|f\|^2 + 2\|u\|^2 - 2\Re\scpr<f,u> = 
    \|f\|^2 = 1,
\end{multline}
and
\begin{equation}
	\|u\|^2 \approx \frac14 \|f\|^2, \qquad \scpr<f,u> = \|u\|^2. 
\end{equation}
Now let us extimate $\|f - g\|$: 
\begin{equation}
    \|f-g\|^2 = \|-u+v\|^2 \approx \|u\|^2 + \|v\|^2 \approx \frac12 \|f\|^2 = \frac12.
\end{equation}
To establish the second statement of the lemma let us divide over $m$ 
left and right side of inequality~\ref{eSpMultEst} in lemma~\ref{lemSpMultEst}, 
and take into accont that we use the same estimates both for $f_1$ and~$f_2$. 
Thus we have to analyze the following inequality: 
\begin{equation}
    \|f-g\|^2 \ge 1 + a^2 - a\sqrt{2}. 
\end{equation}
Using results of the above calculations we have 
\begin{gather}
    \|f-g\|^2 \ge 1 + a^2 - a\sqrt{2}, \\
    \frac12 \ge 2 - \sqrt{2}, \\
    \sqrt{2} \ge \frac32 \label{eqGap}, 
\end{gather}
and we come to contradiction. 
\end{proof}

\begin{defn}
If an automorphism $T$ is approximated by a sequence of icebergs $\ice_n$ 
and ${\mu(\cup\iceberg_n) \to \beta}$ as ${n \to \infty}$ then 
we say that $T$ has $\beta$-local iceberg rank. 
\end{defn}

\begin{rem}
The method of lemma~\ref{commonLemSimplicity} works for $T$ if 
\begin{equation}
	\beta_I(T) > 8/9. 
\end{equation}
The reason is a~small gap in equation~\eqref{eqGap}. 
\end{rem}

\begin{proof}
Suppose that $\beta_I(T) \ge \beta$ and consider a sequence of icebergs 
with approximation property and ${\mu(\cup\iceberg_n) \to \beta}$. 
Let us define $g$ in the same way as in the proof of lemma \ref{commonLemSimplicity}, 
and assume that ${\|f\| = 1}$. Asymptotically we have 
\begin{equation}
	\|f-g\|^2 \le \|u\|^2 + \|v\|^2 + (1-\beta)\|f\|^2 \approx \left( \frac12\beta + 1-\beta \right) \|f\|^2, 
\end{equation}
since we have no information about $f$ outside $\cup\iceberg_n$ except 
${\|f|_{\cup\iceberg_n}\|^2 \to \beta \|f\|^2}$ (in force of ergodic theorem). 
Further, ${a^2 = \|g\|^2 \approx \beta\|f\|^2}$, and the following two values must be compared 
\begin{equation}
	1 - \frac12\beta \quad \text{and} \quad 
	1 + a^2 - a \sqrt2 \approx 1 + \beta - \sqrt{2\beta}. 
\end{equation}
Solving equation 
\begin{equation}
	1 - \frac12\beta = 1 + \beta - \sqrt{2\beta} 
\end{equation}
we find the critical value of $\beta = 8/9$. 
\end{proof}

\begin{quest}
Is it possible to find $T$ with non-simple spectrum and ${\beta_I(T) \le 8/9}$\,?
\end{quest}

\subsection{Simplicity of spectrum for randomized iceberg transformations}

\begin{figure}[th]
  \centering
  \unitlength=1mm
  \includegraphics{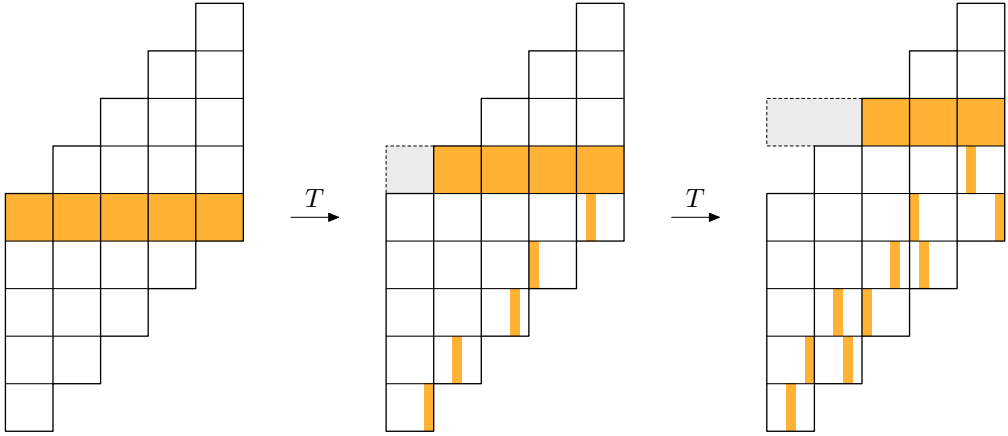} 
  \caption{Iterates of the base level of an iceberg.} 
  \label{fRandIcbLevelDyn}
\end{figure}

\begin{lem}\label{lemSimpleSpRandom}
Suppose that 
\begin{equation}\label{eqJumpsRandom}
	\left\| 
		\mu( TV_{n,k,k-1} \mid V_{n,s,s-h}) - \frac1{h_n} 
	\right\|_1 
	= o\left( \frac1{h_n} \right). 
\end{equation}
Then $\scpr<u,v> \to 0$ and $\scpr<f,v> \to 0$ as $n \to \infty$. 
\end{lem}

This lemma immediately implies that $T$ has simple spectrum (see lemma~\ref{commonLemSimplicity}). 

\begin{proof}
It is not hard to see that without loss of generality 
we can assume the following condition: 
${\mu( TV_{n,k,k-1} | V_{n,s,s-h}) - h_n^{-1} \equiv 0}$. 
Since ${q_n \to \infty}$ we can restrict $f$ to the body~$\icebody_n$, 
where ${\mu(\cup\icebody_n \mid \cup\iceberg_n) \to 1}$. 
Suppose that $f$ is $\ice_{n-1}$-measurable. 
Then for any top set in a fat column $V_{n,k,k-1}$ we see that 
$T^jV_{n,k,k-1}$ intersects the sequence of sets $V_{n,s,s-h+j}$, and 
\begin{equation} 
	\mu(T^jV_{n,k,k-1} \cap V_{n,s,s-h+j}) = \frac1{h_n^2}. 
\end{equation}
Hence, the function $v$ takes value $f_{(n)}(j)$ on each set $V_{n,s,s-h+j}$. 
At~the~same time 
\begin{equation} 
	f|_{V_{n,s,s-h+j}} \equiv f_{(n)}(s-h+j), 
\end{equation}
and 
\begin{equation} 
	\scpr<v|_{T^jV_{n,k,k-1}},f> = 
	\sum_{s=1}^{h_n} f_{(n)}(j) \, f_{(n)}(s-h+j) \cdot \frac1{h_n^2} =   
	O\left( h_{n-1} \cdot \frac1{h_n^2} \right) = 0. 
\end{equation}
Here we use the property ${\sum_{\kappa=0}^{h_{n-1}-1} f_{(n-1)}(\kappa) = 0}$ 
which is invariant under rotation map ${t \mapsto t+1}$, hence, the sequence 
\begin{equation} 
	(f_{(n)}(s-h+j))_{s=1}^{h_n-1} = (f_{(n)}(1-h+j),\ldots,f_{(n)}(j))
\end{equation}
is split into blocks of length $h_{n-1}$ such that 
the sum of the elements in each group equals zero. 
\end{proof}

\begin{thm}
Iceberg transformations given by uniformly distributed on $\Set{Z}_{h_n}$ 
i.i.d.\ random rotations have simple spectrum almost surely, 
if ${q_n/h_n \to \infty}$. 
\end{thm}

\begin{proof}
To prove the theorem it is enough to apply lemma~\ref{lemSimpleSpRandom} 
and to observe that estimate \eqref{eqJumpsRandom} follows from independence 
of $\a_{n,y}$ and~$\a_{n,y+1}$. 
\end{proof}

\section{Correlation decay}

\def\Di{\D}

In this section we discuss estimation of correlation decay for 
iceberg transformations with uniform i.i.d.\ rotations~$\rho_{\a_{n,y}}$ 
(statements (i) and~(ii) of theorem~\ref{thmRandomIcebergMap}). 
A~detailed proof is included in the second part of the paper. 

\begin{rem}
Notice that Ornstein rank one transformations is a case of iceberg transformations 
with random rotations $\a_{n,y}$ distributed on an interval $[0,\a_n^{\mathrm{max}}]$ 
and satisfying restriction ${\a_n^{\mathrm{max}} \ll h_n}$. 
\end{rem}

Let us recall that correlations associated with a measure preserving transformation~$T$ 
is given by the formula
\begin{equation}
	R(t) = \scpr<T^t f,g> = \int_X f(x) \,\bar g(x) \,d\mu, 
\end{equation}
where $f,g \in L^2(X,\mu)$ are fixed. 
One can observe that $R(t)$ is approximated by the following functions 
\begin{equation}
	\CyR_n(t) = \frac1{h_n} \sum_{j \in \Set{Z}_{h_n}} f_{(n)}(j) \,\bar g_{(n)}(j-t) 
\end{equation}
if the sequence of icebergs $\iceberg_n$ are uniform (measure of any fat column is close to~$h_n^{-1}$). 
Let us estimate $\CyR_{n+1}(t)$ for ${t = sh_n}$, ${1 \le s \le q_n-1}$. 
This assumption does not influence to the merits of case but seriously simplifies calculations. 
The common case is studied in the detailed proof. We have 
\begin{equation}
	\CyR_{n+1}(sh_n) = \frac1{h_{n+1}} \sum_{j \in \Set{Z}_{h_{n+1}}} f_{(n+1)}(j) \,\bar g_{(n+1)}(j-sh_n) = 
	\frac1{q_n} \sum_{y=0}^{q_n-1} \CyR_n(\a_{n,y}-\a_{n,y-s}). 
\end{equation}
The sequence $(\a_{n,y})_{y=0}^{q_n-1}$ is i.i.d.\ and we can expect 
that ${\a_{n,y}-\a_{n,y-s}}$ is a mixing process with fast rate of mixing. 
Well, in fact this sequence is almost non-correlated and it turns out that 
this property is enough to see the desired effect. 
Since $\CyR_{n+1}(sh_n)$ is similar to a sum of i.i.d.\ random variables, we get 
\begin{equation}
	|\CyR_{n+1}(sh_n)| \sim \frac1{\sqrt{q_n}} \sqrt{\Di \CyR_n(t)}. 
\end{equation}
In other terms, 
\begin{equation}
	\Di \CyR_{n+1}(sh_n) \sim \frac1{h_{n+1}}(h_n\Di \CyR_n(t) + (q_n-1)h_n\frac1{q_n} \Di \CyR_n(t)) \sim 2\Di \CyR_n(t), 
\end{equation}
and iterating this estimate and uning equality ${h_{n+1} = q_nh_n}$ we have 
\begin{equation}
	\Di \CyR_n(t) = O\left( 2^n \frac1{h_n} \right), \qquad n \to \infty. 
\end{equation}
Thus, for $t \sim h_n$ we have $R(t) = O(2^n h_n^{-1/2}) = O(t^{-1/2+\eps})$ for any ${\eps > 0}$ 
(statement (iii) in theorem~\ref{thmRandomIcebergMap}). 

\begin{rem}
One can observe that the proof is based on ``simple effects''. 
We have used only momnt estimations or, in other words, we deal with 
random sums in a Hilbert space, and central limit theorem is {\it not\/} used throughout the proof. 
At the same time there are no indications to new effects if we apply CPT. 
Probably, this estimate cannot be improved. 

Further, notice that 
the estimate $R(t) \sim O(t^{-a+\eps})$ with ${a < 1/2}$ 
whould imply ${R \in L^2(\Set{Z})}$ and Lebesgue spectrum of~$T$, 
and $1/2$ is the optimal value of~$a$ in the estimate. Nevertheless, 
observe that $\sigma$ can be absolutely continuous even in the case when
estimate ${R(t) \sim O(h_n^{-1/2+\eps})}$ is the best possible. 
Moreover, one can expect that for a typical $\ice_n$-measurable function~$f$ 
the density of a~hypotetical Lebesgue spetctral measure $\sigma_f$ has singularities. 
Such effect is observed for densities of spectral measures for rank one flows studied in~\cite{LebesgueFlows}. 
\end{rem}

\begin{rem}
Let $\sigma$ be a spectral measure corresponding to an $\ice_n$-measurable function~$f$. 
Then ${\sigma \conv \sigma \ll \la}$ since it has $L^2$ Fourier coefficients, ${R^2 \in L^2(\Set{Z})}$. 
Thus, ${\sigma = \sigma_{ac} + \sigma_s}$, where ${\sigma_{ac} \ll \la}$, ${\sigma_s \perp \la}$, 
and $\supp \sigma_s$ is not a semi-Kronecker (see~\cite{FerencziLem}). 
%
\end{rem}

%

\section{Acknowledgments}

The author is very gratefull to 
V.\,Ryzhikov, A.\,A.\,St\"epin, S.\,V.\,Konyagin, J.\,Bourgain, 
R.\,Grigorchuk and T.\,Erdelyi 
for the interest to this investigation and helpful remarks. 
The author is very grateful to 
A.\,M.\,Vershik, K.\,Petersen and B.\,Weiss 
for the helpful discussions and 
remarks concerning adic systems and symbolic dynamics. 
The author would like to express special gratitude to 
El~H.~El~Abdalaoui, M.\,Lemanczyk and J.\,P.\,Thouvenot 
for the careful examination of the geometric aspects of iceberg transformations 
contributed to sections \ref{sIcebergAtAGlance} and~\ref{sIcebergMapDef} of the paper. 
%
%
%
I~would like to thank B.\,M.\,Gurevich, V.\,I.\,Oseledec, S.\,Pirogov and all the partisipants 
of seminar ``Ergodic theory and statistical mechanics'' at Lomonosov Moscow State University 
for the attention to this research. 

\medskip
The work is supported by grant for leading Russian scientific schools \No\,NSh-3038.2008.1

\bibliographystyle{plain}
\bibliography{IcePaperI}
 
\end{document}